\documentclass[11pt]{amsart}

\usepackage{geometry, amsmath, amssymb, amsthm, mathrsfs, setspace, comment, amscd, latexsym, mathtools, bm}
\usepackage[all]{xy}                          
\usepackage{parskip}
\usepackage[abs]{overpic} 
\usepackage{pict2e}

\usepackage{xcolor,graphicx}

\usepackage{hyperref}
\hypersetup{
    colorlinks=true,
    citecolor=blue,
    linkcolor=blue,
    urlcolor=blue,
}

\makeatletter
\@addtoreset{equation}{section}

\makeatother

\newtheorem{theorem}{Theorem}[section]
\newtheorem{lemma}[theorem]{Lemma}
\newtheorem{corollary}[theorem]{Corollary}
\newtheorem{proposition}[theorem]{Proposition}

\theoremstyle{definition}
\newtheorem{definition}[theorem]{Definition}

\newtheorem{convention}[theorem]{Convention}

\newtheorem{remark}[theorem]{Remark}

\newcommand{\del}{\partial}

\newcommand{\Z}{\mathbb{Z}}
\newcommand{\R}{\mathbb{R}}
\newcommand{\C}{\mathbb{C}}
\newcommand{\CP}{\mathbb{CP}}

\newcommand{\M}{\mathcal{M}}
\renewcommand{\H}{\mathcal{H}}

\newcommand{\Crit}{\mathrm{Crit}}
\newcommand{\Critv}{\mathrm{Critv}}

\renewcommand{\Im}{\mathrm{Im}}
\renewcommand{\O}{\mathcal{O}}

\title[Symplectic submanifolds from hyperelliptic Lefschetz fibrations]{Symplectic submanifolds in dimension $6$ from hyperelliptic Lefschetz fibrations}
\author[Takahiro Oba]{Takahiro Oba}
\address{Department of Mathematics, The University of Osaka, 1-1 Machikaneyama, Toyonaka, Osaka 560-0043, Japan}
\email{taka.oba@math.sci.osaka-u.ac.jp}
\date{\today}

\begin{document}

\maketitle

\begin{abstract}
We provide a closed, simply connected, symplectic $6$-manifold having infinitely many codimension $2$ symplectic submanifolds. These are mutually homologous but homotopy inequivalent, and furthermore, they cannot admit complex structures. 
The key ingredient for the construction is hyperelliptic Lefschetz fibrations on $4$-manifolds.   
As a corollary, we present a similar result on symplectic submanifolds of codimension $2$ in higher dimensions. 
In the appendix, we give a proof of the well-known fact that all symplectic submanifolds of codimension $2$ in $(\CP^3, \omega_{\mathrm{FS}})$ of a fixed degree $\leq 3$ are mutually diffeomorphic. 
\end{abstract}

\maketitle

\section{Introduction}

The aim of this paper is to investigate how many connected symplectic submanifolds of codimension $2$, up to some equivalence relation, represent a fixed homology class. 
This problem is motivated by a well-known fact in complex geometry (\cite[Section 1]{FS}, \cite[Remark 2]{Totaro}): 
for a simply connected compact K\"{a}hler manifold $X$, its complex hypersurfaces representing a fixed homology class $\alpha \in H_{2n-2}(X; \Z)$ are unique up to smooth isotopy, especially up to diffeomorphism. 
Indeed, the cohomology class $\mathrm{PD}(\alpha) \in H^{2}(X; \Z)$ Poincar\'e dual to $\alpha$ can be realized as the first Chern class of a unique holomorphic line bundle $L_{\alpha}$ over $X$, if exists, and smooth complex hypersurfaces representing $\alpha$ correspond to transverse holomorphic sections of $L_{\alpha}$. 
Such sections form a connected subspace in the projective space $\mathbb{P}(H^0(X, L_{\alpha}))$, so one can find a smooth isotopy between two given hypersurfaces in question. 
Note that when $X$ is a complex surface, the adjunction formula completely determines the topology of complex curves representing $\alpha$; 
hence, it is interesting to ask about smooth isotopy classes rather than just diffeomorphism classes in this dimension. 

In symplectic geometry, finiteness of symplectic representatives has been studied particularly in dimension $4$. 
The adjunction formula restricts the diffeomorphism types of symplectic submanifolds in symplectic $4$-manifolds, which makes it interesting to consider the problem on representatives up to smooth isotopy as in the complex case. 
One of the celebrated results in this direction is Fintushel and Stern's construction \cite{FS} of an infinite family of homologous but smoothly non-isotopic symplectic surfaces in a simply connected $4$-manifold: see also \cite{Smith01}, \cite{EtPa04}, \cite{Vid}, \cite{EtPa05}, \cite{EtPa06}, \cite{PaPoVi}, \cite{HaPa}, \cite{EtPa08}. 

In higher dimensions, Auroux \cite{Aur} showed that for a given integral symplectic manifold $(M, \omega)$, symplectic submanifolds of $(M,k\omega)$ derived from Donaldson's construction \cite{Do} are unique up to symplectic isotopy for a sufficiently large $k>0$. 
As for non-uniqueness, the aforementioned results in $4$-dimension yield smoothly non-isotopic homologous symplectic submanifolds in higher dimensions.  
Indeed, take an infinite family $\{\Sigma_n\}_{n\in \Z_{>0}}$ of smoothly non-isotopic symplectic surfaces in a $4$-dimensional symplectic manifold $(X_1,\omega_1)$ which represent the same homology class and, for example, can be distinguished by the fundamental group of their complements (see \cite{EtPa08} and \cite{PaPoVi} for such an example). 
Then, in the Cartesian product of $(X_1, \omega_1)$ and an arbitrary symplectic manifold $(X_2, \omega_2)$, the family $\{\Sigma_n \times X_2\}_{n \in \Z_{>0}}$ provides the desired one. 

This turns the problem of symplectic representatives in higher dimensions more interesting up to diffeomorphism or even up to homotopy equivalence than up to smooth isotopy.
For some special cases, diffeomorphism types of representatives are known to be unique (see Appendix \ref{appendix}). 
Our main result shows that, in general, the finiteness of homotopy equivalence classes of symplectic representatives fails in dimension $6$. 

\begin{theorem}\label{thm: main}
There exists a simply connected closed symplectic $6$-manifold containing infinitely many connected symplectic submanifolds. 
These submanifolds are mutually homologous but homotopy inequivalent, and furthermore, they cannot support complex structures. 
\end{theorem}

The next corollary immediately follows from the above theorem. 

\begin{corollary}\label{cor: main}
There exists a simply connected closed symplectic manifold of dimension greater than $4$ containing infinitely many connected symplectic submanifolds which are mutually homologous but homotopy inequivalent.
\end{corollary}

The proof of Theorem \ref{thm: main} is based on the results of Siebert and Tian \cite{ST} and Fuller \cite{Fuller}, who proved that up to blow-up a hyperelliptic Lefschetz fibration is a double cover of a rational ruled surface branched along a smooth embedded surface. 
Given a $4$-manifold $M$ admitting a hyperelliptic Lefschetz fibration, by following the method of Siebert and Tian, we first construct a $4$-dimensional submanifold $Y$ in a blow-up $X$ of a $\CP^2$-bundle over $S^2$ in such a way that $Y$ is diffeomorphic to a blow-up of the given $M$. 
Then, we equip the $6$-manifold $X$ with a symplectic form $\omega$ for which $Y$ is symplectic. 
To produce the desired infinite family of symplectic submanifolds, one could apply this construction to an infinite family of hyperelliptic Lefschetz fibrations. 
But the problem is that even if one gets a family of homologous associated submanifolds $Y_n$ in a common $6$-manifold $X$, it is uncertain to have obtained one symplectic form on $X$ for which all $Y_n$'s are simultaneously symplectic.  
As a remedy for this problem, we shall introduce a twisting operation for the submanifold $Y \subset (X,\omega)$ associated to a Lefschetz fibration which preserves $(X, \omega)$ but changes the topology of $Y$. 
This can be seen as a relative version of the braiding construction introduced by Auroux, Donaldson and Katzarkov \cite{ADK}. 

Note that classically, a non-trivial genus-$1$ Lefschetz fibration can be embedded in the trivial $\CP^2$-bundle over $S^2 \cong \CP^1$, for example, by 
\begin{align*}
\{((a:b), (x:y:z)) \in \CP^1 \times \CP^2 \mid a^np_1(x,y,z)=b^np_2(x,y,z)\}
\end{align*}
 for some $n \in \Z_{>0}$ and generic cubic polynomials $p_1(x,y,z)$ and $p_2(x,y,z)$. 
Our construction is different from this algebraic construction: in fact, for the elliptic case, we will use a family of quartic polynomials.
We also would like to refer to a result by Ghanwat and Pancholi \cite{GhPa}, who showed that every closed orientable $4$-manifold can be smoothly embedded into $\CP^3$ by, similarly to us, employing fibration-like structures in the smooth category (see also \cite{GanNaSa} and \cite{PanPre} for related results). 

The outline of this paper is as follows:  
in Section \ref{section: preliminaries}, we review basic material related to this paper such as mapping class groups, braid groups and Lefschetz fibrations. 
We also give a brief explanation of K\"ahler forms on complex blow-ups. 
In Section \ref{section: submanifold}, we construct a $4$-dimensional submanifold $Y$ in a closed $6$-manifold $X$ for a given hyperelliptic Lefschetz fibration. 
We also discuss symplectic forms on $X$ for which $Y$ is symplectic and examine the homology class $[Y] \in H_4(X;\Z)$.  
Section \ref{section: main thm} contains the proof of Theorem \ref{thm: main}. 
We first introduce an operation for symplectic submanifolds constructed in Section \ref{section: submanifold} to change its topology and then give an infinite family of hyperelliptic Lefschetz fibrations. 
We finally prove the main theorem and its corollary at the end of this section using results shown in this paper. 
The last section is an appendix on the diffeomorphism types of symplectic submanifolds in $\CP^3$ with the Fubini--Study form $\omega_{\mathrm{FS}}$ of degree $\leq 3$. 

\section{Preliminaries}\label{section: preliminaries}

\subsection{Braid and mapping class groups}\label{section: braids}

This section will give a brief review of braid groups and mapping class groups. 
For more information, we refer the reader to \cite{FM}. 

Let $\mathrm{Conf}(F,n)$ be the configuration space of unordered $n$ distinct points in a smooth surface $F$. 
If $F$ is connected,  $\mathrm{Conf}(F,n)$ is path-connected, and hence the fundamental group $\pi_1(\mathrm{Conf}(F,n), [\bm{x}_0])$ is independent of the choice of the basepoint $[\bm{x_0}] \in \mathrm{Conf}(F,n)$ up to isomorphism. 
In this sense, let $B(F, n)$ denote $\pi_1(\mathrm{Conf}(F,n), [\bm{x}_0])$ for a connected surface $F$, omitting the basepoint.
If $F=\C$, the group $B(\C, n)$ is called the \textit{braid group} on $n$ strands. 
Set $\bm{x}_{0}=(x_1,\ldots, x_n)$ for the basepoint of the fundamental group $\pi_1(\mathrm{Conf}(\C,n), [\bm{x}_0])$.
The standard generators $\sigma_i$'s ($i=1,\ldots,n-1$) for $B(\C,n)$, known as the Artin generators, are loops in $\mathrm{Conf}(F,n)$ based at $[\bm{x}_0]$ which correspond to a motion switching $x_i$ and $x_{i+1}$ in counterclockwise direction, as in Figure \ref{fig: sigma_i}, and fixing the other $n-2$ points. 
Consider the case where $F$ is a $2$-sphere $S^2$, which we shall regard as the one-point compactification of $\C$. 
The group $B(S^2, n)$ is called the \textit{spherical braid group} on $n$ strands. 
In view of the inclusion $\C \hookrightarrow S^2$, we have a natural homomorphism $B(\C, n) \rightarrow B(S^2, n)$. 
Let $\hat{\sigma}_{i}$ denote the image of each $\sigma_{i}$ under this homomorphism. 
The group $B(S^2, n)$ has the following presentation: 
\begin{align*}
	\langle \ \hat{\sigma}_1, \ldots, \hat{\sigma}_{n-1} 
    \mid \ & \hat{\sigma}_i \hat{\sigma}_{i+1} \hat{\sigma}_i  =  \hat{\sigma}_{i+1} \hat{\sigma}_{i} \hat{\sigma}_{i+1} \mathrm{\ for \ all\ } i, \\
	&  \hat{\sigma}_i \hat{\sigma}_j  =  \hat{\sigma}_j \hat{\sigma}_i \mathrm{\ for \ } |i-j|>1, \\
	 & \hat{\sigma}_1  \cdots \hat{\sigma}_{n-1} \hat{\sigma}_{n-1} \cdots  \hat{\sigma}_{1}   =  1 \ \rangle. 
\end{align*} 

\begin{figure}[h]
	\centering
	\begin{overpic}[width=100pt,clip]{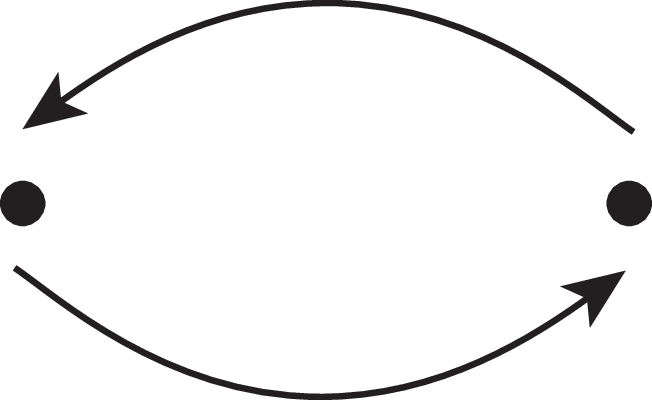}
 	\put(-14,28){$x_{i}$}
	\put(104,28){$x_{i+1}$}
	\end{overpic}
	\caption{Generator $\sigma_i$.}
	\label{fig: sigma_i}
\end{figure}

Next, let $F$ be a closed, connected, oriented, smooth surface possibly with boundary and $P$ the set of $n$ marked points on the interior of $F$.  
The \textit{mapping class group} $\M(F,n)$ of $F$ is the group of isotopy classes of orientation-preserving diffeomorphisms of $F$ which are supported in the interior of $F$ and which preserve $P$ setwise. 
When $P$ is the empty set, we omit the number of elements of $P$, namely $0$ from the notation. 
Although the composition of diffeomorphisms defines a natural product structure on $\M(F,n)$, a different group structure on it introduced below is suitable for our purpose. 

\begin{convention}\label{conv: product}
For $\varphi_1, \varphi_2 \in \M(F, n)$, a product of $\varphi_1$ and $\varphi_2$ is defined by 
\[
	\varphi_1 \varphi_2 \coloneqq  \varphi_2 \circ \varphi_1. 
\]
\end{convention}

It is easy to check that the new group structure is isomorphic to the initial one via the inverse map $\varphi \mapsto \varphi^{-1}$. 

If $F$ is a $2$-sphere $S^2$ with $n$ marked points, there is a homomorphism $\psi \colon B(S^2, n) \rightarrow \M(S^2, n)$, whose kernel is generated by the \textit{full-twist} $(\hat{\sigma}_1 \cdots \hat{\sigma}_{n-1})^n$. 
One can see that the image of this element is the Dehn twist along a simple closed curve surrounding all marked points, and $(\hat{\sigma}_1 \cdots \hat{\sigma}_{n-1})^{2n}=1$ in $B(S^2, n)$. 
In fact, the mapping class group $\M(S^2, n)$ has the presentation 
\begin{align*}
	\langle \  \overline{\sigma}_1, \ldots, \overline{\sigma}_n 
	 \mid \ \ & \overline{\sigma}_i \overline{\sigma}_{i+1} \overline{\sigma}_i  =  \overline{\sigma}_{i+1} \overline{\sigma}_{i} \overline{\sigma}_{i+1} \mathrm{\ for \ all\ } i, \\
	& \overline{\sigma}_i \overline{\sigma}_j  =  \overline{\sigma}_j \overline{\sigma}_i \mathrm{\ for \ } |i-j|>1, \\
	 &  \overline{\sigma}_1  \cdots \overline{\sigma}_{n-1} \overline{\sigma}_{n-1} \cdots  \overline{\sigma}_{1}   =  1, \\
	 &  (\overline{\sigma}_1  \cdots \overline{\sigma}_{n-1})^{n}=1 \ \rangle,  
\end{align*}
where each $\overline{\sigma}_i$ is the image of $\hat{\sigma}_i$ under the homomorphism 
\[
	\psi \colon B(S^2, n) \rightarrow \M(S^2, n). 
\]

Now suppose that $F$ is a closed, connected, oriented, smooth surface of genus $g$. 
Let $\iota$ denote the mapping class of an involution of $F$ with $2g+2$ fixed points. 
The centralizer $\H(F)$ of $\iota$ in $\M(F)$ is called the \textit{hyperelliptic mapping class group} of $F$. 
The group $\H(F)$ has a connection to $\M(S^2, 2g+2)$: 
the involution $\iota$ of $F$ yields a double branched covering $F \rightarrow S^2$ whose branch set consists of $2g+2$ points. 
This covering induces a surjective homomorphism 
\begin{equation} 
	\phi \colon \H(F) \rightarrow \M(S^2, 2g+2)
\end{equation} 
with the kernel generated by $\iota$. 
Let $c$ be a simple closed curve on $F$ which is preserved by $\iota$. 
Then, it is well known that the image $\overline{\tau}_{c} \coloneqq \phi(\tau_c)$ of the Dehn twist $\tau_c$ along $c$ is a conjugate of $\overline{\sigma}_1$ by some element if $c$ is a non-separating curve; it is a conjugate of $(\overline{\sigma}_1 \cdots \overline{\sigma}_{2h})^{4h+2}$ by some element if $c$ is a separating curve bounding a genus-$h$ subsurface of $F$. 
One can choose a distinguished lift $\hat{\tau}_c$ of $\overline{\tau}_c$ to $B(S^2, 2g+2)$ as follows. 
In principle, there are two lifts of each element of $\M(S^2, 2g+2)$. 
A lift of a conjugate of $\overline{\sigma}_1$ (resp. $(\overline{\sigma}_1 \cdots \overline{\sigma}_{2h})^{4h+2}$), however, is unique once we fix its lift to be $\hat{\sigma}_1$ (resp. $(\hat{\sigma}_1 \cdots \hat{\sigma}_{2h})^{4h+2}$) because $B(S^2, 2g+2)$ is a central extension of $\M(S^2, 2g+2)$. 
Hence, the distinguished lift $\hat{\tau}_c$ is uniquely determined.

\subsection{Lefschetz fibrations}\label{section: LF}

We next briefly review Lefschetz fibrations (see \cite[Section 8]{GS} for details). 
Let $M$ be a compact, connected, oriented, smooth $4$-manifold and $S$ a $2$-sphere $S^2$ or a closed $2$-disk $D^2$. 
\begin{definition}
A smooth map $f \colon M \rightarrow S$ is called a \textit{Lefschetz fibration} if it satisfies the following conditions: 
\begin{enumerate}
\item for each critical point $p$ of $f$, there exist local complex coordinates $(z_1,z_2)$ centered at $p$ (resp. a local complex coordinate $w$ centered at $f(p)$) compatible with the orientation of $M$ (resp. $S$) such that $f$ is locally written as 
\[
	w=f(z_1,z_2)=z_1^2+z_2^2\, ; 
\]
\item the map $f$ is injective on the set of critical points; 
\item $f^{-1}(\del S)=\del M$;
\item if $S=D^2$, the critical values of $f$ lie in the interior of $D^2$ and the restriction $f|_{f^{-1}(\del D^2)}$ has no critical points; 
\item no fiber contains a $(-1)$-sphere. 
\end{enumerate}
\end{definition}

Given a Lefschetz fibration $f \colon M \rightarrow S$, we say a fiber of $f$ is \textit{singular} if it contains a critical point of $f$; otherwise, it is said to be \textit{regular}. 
The regular fibers are diffeomorphic to a smooth closed orientable surface. 
Moreover, since $f$ is defined over $S^2$ or $D^2$, all regular fibers are connected by \cite[Proposition 8.1.9]{GS}. 
The \textit{genus} of $f$ is defined to be that of a regular fiber. 
A singular fiber of $f$ is called \textit{irreducible} if the complement of the critical point is connected; otherwise, it is called \textit{reducible}. 

For a Lefschetz fibration $f \colon M \rightarrow S$, let $\Crit(f)$ (resp. $\Critv(f)$) denote the set of critical points (resp. values) of $f$ and let $F_{b_{0}}$ be the fiber of $f$ over a base point $b_0$. 
Since $f$ is a smooth fiber bundle when restricted to $M \setminus f^{-1}(\Critv(f))$ we have the monodromy representation $\rho \colon \pi_1(S^2 \setminus \Critv(f), b_0) \rightarrow \M(\Sigma_g)$ of this smooth fiber bundle with respect to a fixed identification of $F_{b_0}$ with a closed Riemann surface $\Sigma_g$ of genus g.  
Notice that thanks to Convention \ref{conv: product}, $\rho$ becomes a homomorphism. 
We choose generators for $\pi_1(S^2 \setminus \Critv(f), b_0)$ as follows: 
set $\Critv(f)=\{b_1, \ldots, b_k\}$. 
Take smooth embedded paths $\gamma_1, \ldots, \gamma_k \colon [0,1] \rightarrow S^2$ based at $b_0$ such that $\gamma_j(1)=b_j$ and $\gamma_j^{-1}(\Critv(f))=\{1\}$ for each $j$; they are disjoint except at $b_0$ and appear in order by traveling counterclockwise around $b_0$.  
To each $\gamma_i$ is associated a loop $\hat{\gamma}_i$ based at $b_0$ which encloses $b_i$. 
We now easily see that $\hat{\gamma}_1, \ldots, \hat{\gamma}_k$ generate $\pi_1(S \setminus \Critv(f), b_0)$. 
The monodromy $\rho(\hat{\gamma}_i)$ is given by the Dehn twist $\tau_{c_{i}}$ along a simple closed curve $c_i$ in $F_{b_0}$, called a \textit{vanishing cycle}. 
Note that if $c_i$ is non-separating (resp. separating), the corresponding singular fiber $F_{b_i}$ is irreducible (resp. reducible). 
We notice that $\tau_{c_1} \cdots \tau_{c_k}=\rho(\hat{\gamma}_1  \cdots \hat{\gamma}_k) =1$ as $\hat{\gamma}_1 \cdots \hat{\gamma}_k=1$ when $S=S^2$.
Conversely, given a closed, connected, oriented and smooth surface $\Sigma$ of genus $g$ and an ordered collection of essential simple closed curves $(c_1, \ldots, c_k)$ on $F$ with $\tau_{c_1} \cdots  \tau_{c_k}=1$, 
one can construct a genus-$g$ Lefschetz fibration with the collection of vanishing cycles $(c_1,  \ldots, c_k)$. 
According to a classical result by Kas \cite[Theorem 2.4]{Kas}, the monodromy determines the isomorphism class of a Lefschetz fibration if its genus $g$ is at least $2$; for the case $g=1$, the result still holds true on the assumption that the fibration has at least one singular fiber (see also \cite{Moi} and \cite{Mat}). 
Here, two Lefschetz fibrations $f_1 \colon M_1 \rightarrow S^2$ and $f_2 \colon M_2 \rightarrow S^2$ are said to be \textit{isomorphic} if there are orientation-preserving diffeomorphisms $H \colon M_1 \rightarrow M_2$ and $h \colon S^2 \rightarrow S^2$ such that $h \circ f_1=f_2 \circ H$. 

In this paper, we mainly consider a special class of Lefschetz fibrations. 
A Lefschetz fibration is said to be \textit{hyperelliptic} if the image of its monodromy representation lies in $\H(\Sigma_g)$. 
Suppose that given a hyperelliptic Lefschetz fibration $f \colon M \rightarrow S^2$ of genus $g$, we have the collection $(c_1, \dots, c_k)$ of vanishing cycles. 
Then, as described in the previous section, to each $\tau_{c_i}$ is associated the distinguished element $\hat{\tau}_{c_i} \in B(S^2, 2g+2)$. 
As $\tau_{c_1} \cdots \tau_{c_k}=1$, we have 
\[
	\hat{\tau}_{c_1} \cdots \hat{\tau}_{c_k}=1 \mathrm{\ or \ } (\hat{\sigma}_1  \cdots \hat{\sigma}_{2g+1})^{2g+2} \mathrm{\ in \ } B(S^2, 2g+2). 
\]
We call this factorization the \textit{global braid monodromy} associated to $f$. 

\begin{remark}
The name of the global braid monodromy is derived from results of Siebert and Tian \cite{ST} and Fuller \cite{Fuller}. 
They showed that, up to blow-up, the total space of a hyperelliptic Lefschetz fibration is the double cover of a rational ruled surface branched along an embedded surface. 
This branch set is braided over $S^2$ with respect to the projection of the ruled surface, so the monodromy we defined corresponds to the one of this braided surface. 
(See \cite[Lemma 1.3]{ST} for details.) 
\end{remark}

\subsection{Symplectic forms on blow-ups}\label{section: Kaehler form}
To prove the main theorem, we need to endow a blow-up of a symplectic $6$-manifold with a symplectic structure. 
In our setup, the initial symplectic form can be always assumed to be K\"ahler near the blow-up locus. 
Thus, here we will give a quick review of K\"ahler forms on complex blow-ups: see \cite[Section 3]{Voisin} for details. 

Let $(X, \omega, J)$ be an almost K\"ahler manifold of dimension $2n$ and $Y$ a compact almost complex submanifold of codimension $k$. 
Suppose that the almost complex structure $J$ is integrable near $Y$. 
Take a collection of holomorphic charts $\{(U_{\alpha}, \varphi_{\alpha})\}_{\alpha}$ covering $Y$ such that 
$\varphi_{\alpha}(U_{\alpha} \cap Y)$ is written as the zero set $\{ x \in\Im (\varphi_{\alpha}) \mid f_{j}^{\alpha}(x)=0, j=1,\cdots, k\}$ by polynomials $f_{j}^{\alpha}$'s. 
Set $V_{\alpha} = \Im (\varphi_{\alpha})$ and define 
\[
	(\tilde{V}_{\alpha})_{Y}=\{([z], x) \in \CP^{k-1} \times V_{\alpha} \mid (f^{\alpha}_1(x), \ldots, f^{\alpha}_k(x)) \in \ell[z]\},
\]
where $\ell[z]$ is a $1$-dimensional subspace of $\C^k$ spanned by the non-zero vector $z \in \C^{k}$. 
It is easy to see that $(\tilde{V}_{\alpha})_{Y}$ is a smooth complex submanifold of $\CP^{k-1} \times V_{\alpha}$. 
The projection $\pi_\alpha \colon (\tilde{V}_{\alpha})_{Y} \rightarrow V_{\alpha}$ to the second factor is a biholomorphism away from $V_{\alpha}\cap \varphi_{\alpha}(Y)$. 
The fiber of $\pi_{\alpha}$ over every point of $V_{\alpha} \cap \varphi_{\alpha}(Y)$ agrees with $\CP^{k-1}$. 
We easily check that each pair of $(\tilde{V}_{\alpha})_{Y}$ and $(\tilde{V}_{\beta})_{Y}$ can be glued together by the matrix $((m_{ij})^{T})^{-1}$ of holomorphic functions defined by $f_j^{\alpha}=\sum_{j=1}^{k}m_{ij}f_i^{\beta}$.
Finally, we define the blow-up $\tilde{X}_{Y}$ of $X$ along $Y$ as the quotient space 
\[
	\tilde{X}_{Y} := (X \setminus Y) \cup (\cup_{\alpha}(\tilde{V}_{\alpha})_{Y})/ \sim
\]
using the above identifications. 
Let $\pi \colon \tilde{X}_Y \rightarrow X$ denote the projection and let $E_{Y}$ be the exceptional divisor. 
Note that away from $E_{Y}$, the projection $\pi$ is a diffeomorphism. 

Next we equip $\widetilde{X}_{Y}$ with a symplectic form. 
We begin by observing that $\pi^{*}\omega$ is degenerate along the kernel of the differential $D\pi|_{E_{Y}}$, away from which in contrast it is clearly non-degenerate. 
To obtain a symplectic form on the whole $\tilde{X}_{Y}$, consider a holomorphic line bundle $L \rightarrow \tilde{X}_{Y}$ corresponding to the divisor $-E_{Y}$. 
As $E_{Y}$ is isomorphic to the projective bundle $\mathbb{P}(N_{Y/X})$ associated to the normal bundle $N_{Y/X}$, the line bundle $L|_{E_{Y}}$ is isomorphic to the hyperplane bundle $\O_{\mathbb{P}(N_{Y/X})}(1)$ over $\mathbb{P}(N_{Y/X})$. 
Moreover, $L$ is trivial when restricted to $ E_{Y} \subset \tilde{X}_{Y}$.
For a hermitian metric $h$ on this bundle induced by a hermitian metric on $N_{Y/X}$, the Chern form $\frac{1}{2\pi i}\del \overline{\del} \log h$ is non-degenerate along the fiber of $\mathbb{P}(N_{Y/X})$. 
By using a partition of unity, one can extend the metric $h$ on $L|_{E_Y}$ to a hermitian metric $h_{L}$ on $L$ to be the flat metric outside a small neighborhood of $E_Y$ for a fixed trivialization of $L$ over $\tilde{X}_{Y} \setminus E_Y$. 
Then, the Chern form $\omega_L$ is a closed $2$-form on $\tilde{X}_{Y}$ supported in a small neighborhood of $E_Y$ and non-degenerate on $\ker(D\pi|_{E_{Y}})$. 
Since $Y$ is compact, for a small $\delta>0$ the $2$-form 
\[
	\pi^{*}\omega + \delta \omega_L
\]
is a symplectic form on $\tilde{X}_{Y}$. 
Notice that by construction a natural almost complex structure on $\tilde{X}_{Y}$ is compatible with this symplectic form.

\section{Submanifolds from hyperelliptic Lefschetz fibrations}\label{section: submanifold}

\subsection{Local models}\label{section: local model}

We begin the construction by producing a submanifold of a $6$-manifold, which is diffeomorphic to the total space of a hyperelliptic Lefschetz fibration with a single singular fiber up to blow-up; this submanifold will play a role of a local model. 
The following argument is inspired by \cite{ST}.

\subsubsection{Description of local models}
Let $D_{\varepsilon}$ be the closed disk in $\C$ centered at the origin of radius $\varepsilon$. 
For an integer $g \geq 2$, we take mutually distinct points $a_3, \ldots, a_{2g+2}$ with $|a_i|>\varepsilon$. 
Set 
\begin{align*}
	B_{\mathrm{irr}}=\left\{(s,(u:v)) \in D_{\varepsilon} \times \CP^1 \; \middle| \; (u^2-sv^2) \cdot \prod_{r=3}^{2g+2}(u-a_rv)=0\right\}
\end{align*}
and 
\begin{align*}
	B_{h,g-h}=\left\{(s,(u:v)) \in D_{\varepsilon} \times \CP^1 \; \middle| \; \prod_{r=1}^{2h+1}\left(u- e^{\frac{2\pi r i}{2h+1}}s^2 v\right) \cdot \prod_{r=2h+2}^{2g+2}(u-a_rv)=0\right\} 
\end{align*}
for any integer $h$ with $0<h \leq g/2$.
The former set $B_{\mathrm{irr}}$ is a non-singular complex curve in $D_{\varepsilon} \times \CP^1$ while the latter set $B_{h}$ is a singular complex curve with only one singular point at $(0, (0:1))$. 
For our purpose, these will be presented as subspaces of $D_{\varepsilon} \times \CP^2$.  
Define $Y^{(0)}_{\mathrm{irr}}$ and $Y^{(0)}_{h,g-h}$ by 
\begin{gather*}
	Y_{\mathrm{irr}}^{(0)}=\left\{(s,(x:y:z)) \in D_{\varepsilon} \times \CP^2 \; \middle| \; (x^2-sz^2) \prod_{r=3}^{2g+2}(x-a_rz)=y^2z^{2g}\right\}, \\
	Y_{h,g-h}^{(0)}=\left\{(s,(x:y:z)) \in D_{\varepsilon} \times \CP^2 \; \middle| \; \prod_{r=1}^{2h+1}\left(x- e^{\frac{2\pi r i}{2h+1}}s^2z\right) \prod_{r=2h+2}^{2g+2}(x-a_rz)=y^2z^{2g}\right\}. 
\end{gather*}
Neither of them is a smooth manifold. 
Indeed, the singular locus of $Y_{\mathrm{irr}}^{(0)}$ is $D_{\varepsilon} \times {(0:1:0)}$; that of $Y_{h,g-h}^{(0)}$ is $(D_{\varepsilon} \times {(0:1:0)}) \cup \{(0, (0:0:1))\}$. 
Blowing them up several times along the singular loci as described in Section \ref{section: resolution} below, we finally obtain the smooth proper transforms, denoted by $Y_{\mathrm{irr}}$ and $Y_{h,g-h}$, respectively. 
Note that $Y_{\mathrm{irr}}$ lies in $D_\varepsilon \times (\CP^2 \# (g+1)\overline{\CP}^2)$, and $Y_{h,g-h}$ lies in a blow-up $Bl(D_{\varepsilon}\times (\CP^2 \# (g+1)\overline{\CP}^2))$ of $D_{\varepsilon}\times (\CP^2 \# (g+1)\overline{\CP}^2)$. 
We will specify the diffeomorphism types of $Y_{\mathrm{irr}}$ and $Y_{h,g-h}$ in Section \ref{section: diffeo type}.

\subsubsection{Resolving singularities}\label{section: resolution}

We now explain how to obtain $Y_{\mathrm{irr}}$ and $Y_{h,g-h}$ from $Y^{(0)}_{\mathrm{irr}}$ and $Y^{(0)}_{h,g-h}$, respectively. 

\subsubsection*{The irreducible case. }

First consider the irreducible case. 
As we have observed, the singular points of $Y_{\mathrm{irr}}^{(0)}$ form $D_{\varepsilon} \times \{ (0:1:0) \}$, where we blow up $Y_{\mathrm{irr}}^{(0)}$ first. 
With the holomorphic chart $(\{(s, (x:y:z))\mid y=1\}, (s,x,z))$ of $D_{\varepsilon} \times \CP^2$,  the space $Y_{\mathrm{irr}}^{(0)}$ is written as $\{(s,x,z) \in D_{\varepsilon} \times \C^2 \mid (x^2-sz^2) \prod_{r=3}^{2g+2}(x-a_rz)-z^{2g}=0\}$, in which the singular locus corresponds to $\{(s,0,0) \in D_{\varepsilon} \times \C^2\}$. 
Set $U_0=D_{\varepsilon} \times \C^2$. 
Consider the blow-up map
$\pi_{1}: \tilde{U}_{0}\rightarrow U_0$ of $U_0$ along $\{(s,0,0) \in D_{\varepsilon} \times \C^2\}$ and let $Y_{\mathrm{irr}}^{(1)}$ denote the proper transform of $Y_{\mathrm{irr}}^{(0)}$, where 
\[
	\tilde{U}_0 =\{(s, (\zeta: \eta),(x,z)) \in D_{\varepsilon} \times \CP^{1} \times \C^2 \mid \zeta z= \eta x \}.
\]
Set $U_1=\{(s,(1:u), (v,uv)) \in \tilde{U}_0 \mid (s,u,v) \in D_{\varepsilon} \times \C^2\}$ and take the holomorphic chart $(U_1, (s,u,v))$. 
The preimage $\pi_{1}^{-1}(Y_{\mathrm{irr}}^{(0)})$ is expressed by the following equation on $U_{1}$: 
\begin{align*}
	 v^{2g} \left\{v^2(1-su^2)  \prod_{r=3}^{2g+2}(1-a_r u)-u^{2g} \right\}= 0. 
\end{align*} 
Hence, the proper transform $Y_{\mathrm{irr}}^{(1)}$ is given by 
\begin{align}
	 v^2(1-su^2) \prod_{r=3}^{2g+2}(1-a_r u)-u^{2g}= 0, \label{eqn: first blow-up}
\end{align} 
which is singular along the set $\{(s,(1:0), (0,0))\}$. 
It is easy to see that the singular locus of $Y_{\mathrm{irr}}^{(1)}$ is contained in $U_1$.  
Further blow up $U_1$ along this locus and denote the corresponding projection by $\pi_2 \colon \tilde{U}_1 \rightarrow U_1 \cong D_{\varepsilon} \times \C^2$, where $\tilde{U}_1 \subset D_{\varepsilon} \times \CP^1 \times \C^2$. 
Set $U_{2}=\{(s, (1:p), (q,pq)) \in \tilde{U}_1 \mid (s,p,q) \in D_\varepsilon \times \C^2\}$ and consider the holomorphic chart $(U_2, (s,p,q))$.  
With those coordinates, the proper transform $Y_{\mathrm{irr}}^{(2)}$ of $Y_{\mathrm{irr}}^{(1)}$ is defined by
\begin{align}
	 p^2(1-sq^2) \prod_{r=3}^{2g+2}(1-a_rq)-q^{2g-2} = 0 \label{eqn: second blow-up}
\end{align} 
on $U_2$. 
We see that $Y_{\mathrm{irr}}^{(2)}$ is singular along $\{(s, (1:0), (0,0))\}$ on $U_2$ unless $g=1$ while smooth outside $U_2$.  
Now comparing the equation (\ref{eqn: first blow-up}) with (\ref{eqn: second blow-up}), the induction on $g$ shows that $g+1$ times blowing up gives the smooth proper transform $Y_{\mathrm{irr}}$ of $Y_{\mathrm{irr}}^{(0)}$ in $D_{\varepsilon} \times (\CP^2 \# (g+1)\overline{\CP}^2)$.

\subsubsection*{The reducible case.} 

Next we discuss $Y_{h,g-h}^{(0)}$. 
Recall that it is singular along $D_{\varepsilon} \times \{(0:1:0)\}$ and at $(0, (0:0:1))$. 
The former can be resolved in the same way as $Y_{\mathrm{irr}}^{(0)}$ by $g+1$ point-blow-ups. 
To deal with the latter, take the holomorphic chart $(\{(s,(x:y:z)) \mid z=1\}, (s,x,y))$ of $ D_{\varepsilon} \times \CP^2$ around $(0,(0:0:1))$. 
In light of the identity $\prod_{r=1}^{2h+1}\left(x- e^{\frac{2\pi r i}{2h+1}}s^2 \right)=x^{2h+1}-s^{2(2h+1)}$, with these coordinates $Y_{h,g-h}^{(0)}$ is expressed as
\[
	\left\{(s,x,y) \in D_{\varepsilon} \times \C^2  \ \middle| \  (x^{2h+1}-s^{2(2h+1)}) \prod_{r=2h+2}^{2g+2}(x-a_r)-y^2=0  \right\} 
\]
around the point $(0, (0:0:1))$.
The affine curve $\{(x^{2h+1}-s^{2(2h+1)}) \prod_{r=2h+2}^{2g+2}(x-a_r)=0\}$ defined in a small neighborhood of $(0,0) \in D_\varepsilon \times \C$ has a unique singular point at the origin, which can be resolved by two blow-ups (see \cite[Exercise 7.2.4(b)]{GS} and its answer). 
We will make use of this fact: after two point-blow-ups, we get the proper transform defined by the equation 
\begin{align}
	t^{2(2h-1)}(u^{2h+1}-1) \prod_{r=2h+2}^{2g+2}(t^2u-a_{r})-v^2=0\label{eqn: irr}
\end{align}
in one chart $(t,u,v)$. 
Observe that the singular locus of this affine surface is $\{(0,u,0) \mid u \in \C\}$ and, combining other singularities in other charts, the singular locus of the proper transform is isomorphic to $\CP^1$.
Moreover, it can be checked that the blow-up along $\CP^1$ decreases the exponent of $t$ in front of the first parenthesis in (\ref{eqn: irr}) by $2$ and gives the proper transform the singular locus isomorphic to $\CP^1$ again. 
By the induction on $h$, further $2h-1$ blow-ups along $\CP^1$ desingularize $Y_{h,g-h}^{(0)}$. 
To sum up, the successive $g+3$ blow-ups at points and $2h-1$ blow-ups along $\CP^1$'s yield the smooth proper transform $Y_{h,g-h}$ of $Y_{h,g-h}^{(0)}$.

\subsubsection{Topology of $Y_{\mathrm{irr}}$ and $Y_{h,g-h}$}\label{section: diffeo type}

We shall here specify the diffeomorphism types of $Y_{\mathrm{irr}}$ and $Y_{h,g-h}$. 

\subsubsection*{The irreducible case}
Consider the meromorphic map 
\begin{align*}
D_{\varepsilon} \times (\CP^2 \# (g+1) \overline{\CP}^2) \rightarrow D_{\varepsilon} \times \CP^1
\end{align*}
induced by the map $D_{\varepsilon} \times \CP^2 \dashrightarrow D_{\varepsilon} \times \CP^1$, $(s, (x:y:z)) \mapsto (s,(x:z)) \in  D_{\varepsilon} \times \CP^1$.  
This yields the double branched cover $p_{\mathrm{irr}} \colon Y_{\mathrm{irr}} \rightarrow D_\varepsilon \times \CP^1$ branched along $B_{\mathrm{irr}}$, and the composition $\mathrm{pr}_1 \circ p_{\mathrm{irr}}$ is a genus-$g$ hyperelliptic Lefschetz fibration over $D_{\varepsilon}$ with only one irreducible singular fiber.

\subsubsection*{The reducible case}
In order to examine the reducible case, recall a resolution of the complex curve $B_{h,g-h}$ in $D_{\varepsilon} \times \CP^1$. 
It has an isolated singularity, which can be resolved by two blow-ups. 
Let $\pi_1$ and $\pi_2$ denote the first and second blow-ups, respectively, and $\pi \colon Bl(D_{\varepsilon} \times \CP^1) \rightarrow D_{\varepsilon} \times \CP^1$ their composition. 
The total transform $B'_{h,g-h}=\pi^{-1}(B_{h,g-h})$ of $B_{h,g-h}$ consists of the proper transform $\tilde{B}_{h,g-h}$ of $B_{h,g-h}$ and the two rational curves $\tilde{E}_1$ and $E_2$ with self-intersection numbers $-2$ and $-1$, respectively. 
Here $\tilde{E}_1$ is the proper transform of the exceptional curve $E_1$ of $\pi_1$, and $E_2$ is the exceptional curve of $\pi_2$.
See Figure \ref{fig: resolution} for the case $g=2$ and $h=1$.
The resolution of $B_{h,g-h}$ is essentially the same as that of the affine curve we used in the preceding discussion to obtain $Y_{h,g-h}$, which  leads to a meromorphic map from $Bl(D_{\varepsilon} \times (\CP^2 \# (g+1) \overline{\CP}^2))$ to $Bl(D_{\varepsilon} \times \CP^1)$. 
When restricted to $Y_{h,g-h}$, this map is the double cover $p_{h, g-h} \colon Y_{h,g-h} \rightarrow Bl(D_\varepsilon \times \CP^1)$ branched along $\tilde{B}_{h,g-h}$ and $\tilde{E}_1$. 
The composition $\mathrm{pr}_1 \circ \pi \circ p_{h,g-h}$ is not a Lefschetz fibration but a singular fibration obtained by blowing up at the critical point of a genus-$g$ hyperelliptic Lefschetz fibration $M_{h,g-h}$ with only one reducible singular fiber consisting of the two components of genus $h$ and $g-h$. 
This shows that $Y_{h,g-h}$ is diffeomorphic to $M_{h,g-h} \# \overline{\CP}^2$. 
(Notice that the resulting branched cover $p_{h,g-h}$ is the same as the one given in \cite[Section 1]{ST}.)

\begin{figure}[th]
\vspace{11pt}
		\begin{overpic}[width=300pt,clip]{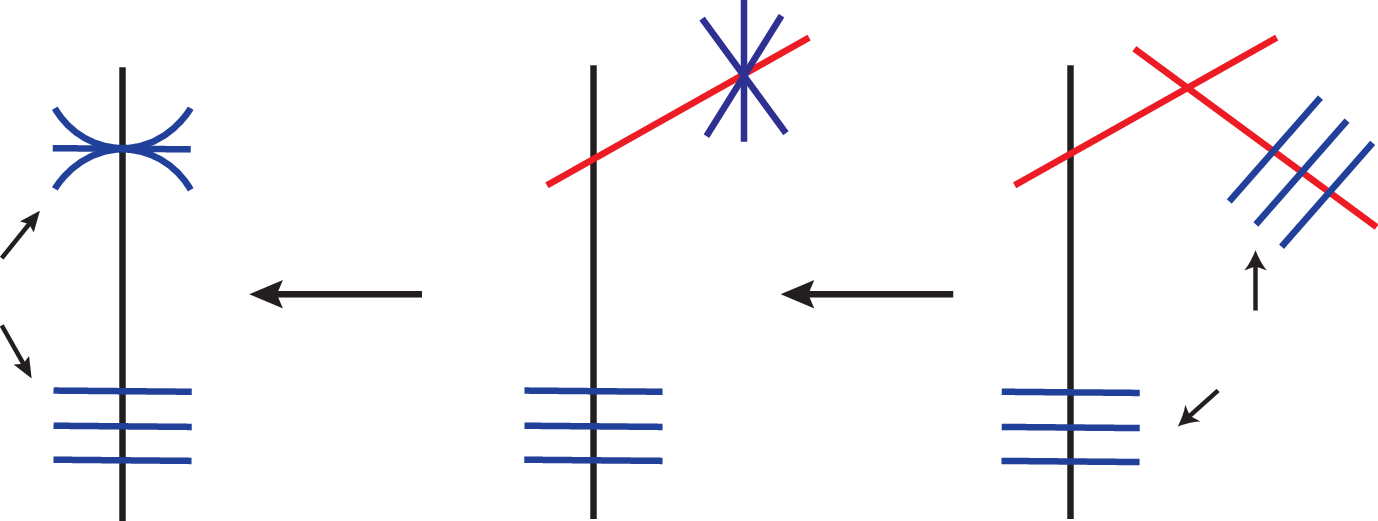}
	\put(15,103){$0\times \CP^1$}
	\put(-20,52){$B_{h,g-h}$}
	\put(265,33){$\tilde{B}_{h,g-h}$}
	\put(106,79){$E_1$}
	\put(105,62){$-1$}
	\put(207,79){$\tilde{E}_1$}
	\put(206,62){$-2$}
	\put(300,67){$E_2$}
	\put(298,54){$-1$}
	\put(70,60){$\pi_1$}
	\put(185,60){$\pi_2$}
	\end{overpic}
	\caption{Total transform $B'_{h,g-h}=\tilde{B}_{h,g-h} \cup \tilde{E}_1 \cup E_2$ of $B_{h,g-h}$ for the case $g=2$ and $h=1$. }
	\label{fig: resolution}
\end{figure}
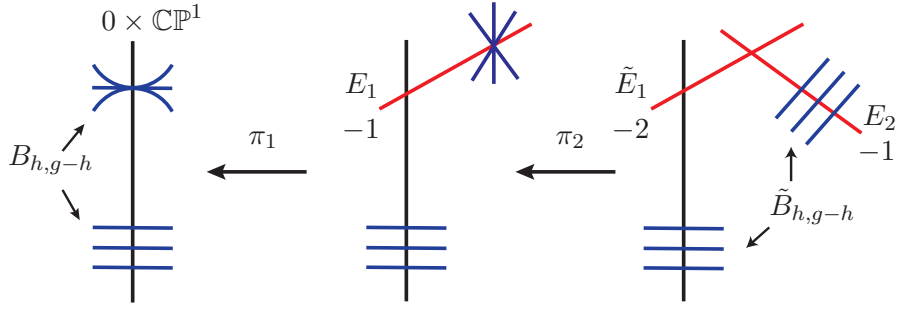

\subsection{Gluing local models}\label{section: gluing}
Now that we have constructed local models, we shall glue them together according to the information about a given hyperelliptic Lefschetz fibration $f \colon M \rightarrow S^2$ of genus $g \geq 2$ or a non-trivial Lefschetz fibration of genus $1$. 
Our construction can be thought as a relative version of Siebert and Tian's construction \cite[Section 1]{ST}. 
Suppose that 
\begin{align*}
\Critv(f)=\{b_1, \ldots, b_k\}.
\end{align*} 
We may assume that $b_j$'s are the $k$-th roots of $1/2$, i.e. $b_j=(e^{2\pi i j}/2)^{1/k}$. 
Divide the base space $S^2$ of $f$ into two disks one of which, say $D_{+}$, contains the all critical values and over the other one of which, say $D_{-}$, the fibration $f$ is trivialized; hence the essential topological information of $f$ is condensed over $D_{+}$. 
In what follows, we will construct a singular space in $D_{+} \times \CP^2$ by gluing local models given in the previous subsection, and then extend it over the whole of a certain $\CP^2$-bundle over $S^2$. 
Finally, blow-ups as in Section \ref{section: resolution} will yield the desired embedded submanifold diffeomorphic to a blow-up of $M$.

To examine the restricted fibration $f|_{f^{-1}(D_{+})}$, let $(c_1, \ldots, c_k)$ be the collection of its vanishing cycles, where each $c_j$ corresponds to $b_j$.
Each $\tau_{c_j}$ descends to the element $\overline{\tau}_{c_j}=\phi(\tau_{c_{j}})$ of $\M(S^2, 2g+2)$ via 
\begin{align*}
\phi \colon \H(S^2, g) \rightarrow \M(S^2, 2g+2), 
\end{align*}
and one can take its distinguished lift $\hat{\tau}_{c_j} \in B(S^2, 2g+2)$ as in Section \ref{section: braids}.
Note that the product of the $\hat{\tau}_{c_j}$'s is either $1$ or $(\hat{\sigma}_1 \cdots \hat{\sigma}_{2g-1})^{2g+2}$ because $\tau_{c_1} \cdots \tau_{c_k}=1$.
Let $a_1, \ldots, a_{2g+2}$ be mutually distinct points in $\C$. 
Regarding $B(S^2, 2g+2)$ as $\pi_1(\mathrm{Conf}(\CP^1, 2g+2), [a_1, \ldots, a_{2g+2}])$ and setting 
\begin{align*}
\Delta \coloneqq \{(z_1, \ldots, z_{2g+2}) \in (\CP^1)^{2g+2} \mid z_{i} \neq z_{j} \mathrm{\ for \ some \ } i \neq j\}, 
\end{align*}
take a smooth path 
\[
\ell_j \colon [0,1] \rightarrow \C^{2g+2} \setminus \Delta \subset (\CP^1)^{2g+2} \setminus \Delta, \quad \ell_j(t)=(a_1^{(j)}(t), \ldots, a_{2g+2}^{(j)}(t))
\] 
in such a way that it starts from $(a_1, \ldots, a_{2g+2})$ and ends at $(a_{\varphi_j(1)}, \ldots, a_{\varphi_j(2g+2)})$ for some element $\varphi_j$ of the symmetric group of order $2g+2$, and it also descends to a loop in $\mathrm{Conf}(\CP^1, 2g+2)$ representing the braid $\hat{\tau}_{c_j}$.

Now remove a small neighborhood of $\Critv(f)$ from $D_{+}$ and identify the resulting holed disk with 
\[
	D_{*}=(D_{1/2}(0) \cup (\cup_{j=1}^{k} D_{3\varepsilon}(b_j) ) )\setminus \cup_{j=1}^{k} D_{2\varepsilon}(b_j) \subset \C, 
\]
where $D_{r}(a)$ denotes the closed disk in $\C$ of radius $r$ centered at $a$. 
Set $a_{i}^{(j)}(t)=a_{i}^{(j)}(0)$ for $t<0$ and $a_i^{(j)}(t)=a_i^{(j)}(1)$ for $t>1$. 
Define the subspace $Y^{(0)}_{+,*}$ of $D_{*} \times \CP^2$ fibered over $D_{*}$ by 
\begin{align*}
	Y^{(0)}_{+,*}   \coloneqq  & \left\{(s, (x:y:z)) \in (D_{1/2}(0)\setminus  \cup_{j=1}^{k} D_{2\varepsilon}(b_j)) \times \CP^2 \ \middle| \
	\prod_{i=1}^{2g+2}(x-a_iz)=y^2z^{2g} \right\}  \\
	& \cup \left(   \bigcup_{j=1}^{k} \bigcup_{t \in [-1/2, 3/2]} \bigcup_{\rho \in [2\varepsilon, 3\varepsilon]} \biggr\{(s, (x: y: z)) \in (D_{3\varepsilon}(b_j)\setminus D_{2\varepsilon}(b_j)) \times \CP^2 \  \middle| \right. \\
	& \hspace{40pt} \left. \left.  s=b_j+\rho e^{\pi (t-1/2)i+i\arg (b_j)}, \prod_{i=1}^{2g+2}(x-a_i^{(j)}(t)z)=y^2z^{2g} \right\} \right).
\end{align*} 
We would like to fill the \textit{holes} by using the local models constructed in Section \ref{section: local model}. 
Let $\ell_{\mathrm{irr}}$ (resp. $\ell_{h,g-h}$) be the path in $(\CP^1)^{2g+2} \setminus \Delta$ defined as the motion of intersections of the family $\{\{\varepsilon e^{2\pi i t}\} \times \CP^1\}_{t \in [0,1]}$ and the complex curve $B_{\mathrm{irr}}$ (resp. $B_{h,g-h}$), which descends to a loop in the configuration space representing the braid monodromy of $B_{\mathrm{irr}}$ (resp. $B_{h,g-h}$). 
Suppose that $c_j$ is a non-separating curve. 
Then, there exists a smooth homotopy $L_j$ between the two paths $\ell_j$ and $\ell_{\mathrm{irr}}$: 
\[
	L_j \colon [0,1] \times [0,1] \rightarrow \C^{2g+2} \setminus \Delta, \quad L_j(t,0)= \ell_j(t), \quad L_j(t,1)= \ell_{\mathrm{irr}}(t).
\]
Applying this homotopy, we glue together $Y_{+, *}^{(0)}$ and the local model $Y^{(0)}_{\mathrm{irr}}$ as follows: 
glue $Y^{(0)}_{\mathrm{irr}}$ and $Y^{(0)}_{+,*} \cap ((D_{3\varepsilon}(b_j) \setminus D_{2\varepsilon}(b_j)) \times \CP^2)$ in $D_{3\varepsilon}(b_j) \times \CP^2$ by interpolating a space in $(D_{2\varepsilon}(b_j) \setminus D_{\varepsilon}(b_j)) \times \CP^2$ defined by $L_j$. 
For the case where $c_j$ bounds a genus-$h$ surface, employ $Y^{(0)}_{h,g-h}$ instead of $Y_{\mathrm{irr}}^{(0)}$: glue $Y^{(0)}_{h,g-h}$ and $Y^{(0)}_{+,*} \cap ((D_{3\varepsilon}(b_j) \setminus D_{2\varepsilon}(b_j)) \times \CP^2)$ in $D_{3\varepsilon}(b_j) \times \CP^2$, filling it by a space associated to a smooth homotopy $L_j$ between $\ell_j$ and $\ell_{h,g-h}$. 
Furthermore, an isotopy of braids helps the resulting glued space in $(D_{*} \cup (\cup_{j=1}^{k} D_{2\varepsilon}(b_j))) \times \CP^2$ to extend to a subspace of $D_{1}(0) \times \CP^2$, say $Y^{(0)}_{+}$, so as to satisfy the following: 
the projection $D_{1}(0) \times \CP^2 \rightarrow D_{1}(0)$ is a fibration over $D_1(0)$ away from $\{b_1, \ldots, b_k\}$ when restricted to $Y^{(0)}_{+}$; over a collar neighborhood $\nu(\del D_{1}(0))$ of $\del D_{1}(0)$, the space $Y^{(0)}_{+}$ agrees with 
\[
	\{(s, (x:y:z)) \in \nu(\del D_{1}(0)) \times \CP^2 \mid x^{2g+2}-z^{2g+2}=y^2z^{2g}\},
\]
if the product $[\ell_1] \cdots [\ell_k] \in B(S^2, 2g+2)$ of the descendent braids is trivial, and otherwise agrees with 
\[
	\{(s, (x:y:z)) \in \nu(\del D_{1}(0)) \times \CP^2 \mid x^{2g+2}-(sz/|s|)^{2g+2}=y^2z^{2g}\}.
\]
Note that the product $[\ell_1] \cdots [\ell_k]$ in $B(S^2, 2g+2)$ is $1$ or $(\hat{\sigma}_1 \cdots \hat{\sigma}_{2g+1})^{2g+2}$ since it lies in the kernel of $\psi \colon B(S^2, 2g+2) \rightarrow \M(S^2,2g+2)$. 

Define $Y^{(0)}_{-}$ to be  
\[
	Y^{(0)}_{-} = \{(s, (x:y:z)) \in D_{1}(0) \times \CP^2 \mid x^{2g+2}-z^{2g+2}=y^2z^{2g}\},   
\]
which corresponds to the total space of the trivial bundle $f|_{D_{-}}$. 
We glue $Y^{(0)}_{+}$ and $Y^{(0)}_{-}$ together. 
If $[\ell_1] \cdots [\ell_k]$ is trivial, we identify $(\del D_1(0) \times \CP^2, \del  Y^{(0)}_{+})$ with $(\del D_1(0) \times \CP^2, \del  Y^{(0)}_{-})$ canonically and put 
\begin{align}\label{def: Y^{(0)}}
	Y^{(0)} \coloneqq Y^{(0)}_{+} \cup Y^{(0)}_{-} \subset S^2 \times \CP^2. 
\end{align}
If $[\ell_1] \cdots [\ell_k]$ is not trivial, we define a map 
\begin{align*}
\Phi_{\mathrm{glue}} \colon (\del D_1(0) \times \CP^2, \del  Y^{(0)}_{-}) \rightarrow (\del D_1(0) \times \CP^2, \del  Y^{(0)}_{+}) 
\end{align*}
by 
$
	\Phi_{\mathrm{glue}}(e^{2\pi i \theta}, (x:y:z))=(e^{-2\pi i \theta}, (e^{-2\pi i \theta}x : e^{-(2g+2)\pi i \theta}y: z))
$
and set 
\[	
	Y^{(0)} \coloneqq Y^{(0)}_{+} \cup_{\Phi_{\mathrm{glue}}} Y^{(0)}_{-} \subset S^2 \tilde{\times}_g \CP^2.
\] 
Here, $S^2 \tilde{\times}_g \CP^2 \coloneqq (D_{1}(0) \times \CP^2) \cup_{\Phi_{\mathrm{glue}}} (D_{1}(0) \times \CP^2)$, and this $\CP^2$-bundle can be non-trivial.  
Indeed, the bundle $S^2 \tilde{\times}_g \CP^2$ is trivial if and only if the loop in $\mathrm{PSU}(3)$ associated to $\Phi_{\mathrm{glue}}$ lifts to a loop in the universal cover $\mathrm{SU}(3)$ of $\mathrm{PSU}(3)$; the latter is equivalent to $g \equiv 1$ modulo $3$. 
We set $X^{(0)}=S^2 \times \CP^2$ if $[\ell_1] \cdots [\ell_k]=1$, and otherwise set $X^{(0)}=S^2 \tilde{\times}_g \CP^2$. 

To complete the construction, we resolve the singular locus of $Y^{(0)}$, which consists of the points corresponding to $(0:1:0) \in \CP^2$ in the fibers and the singular points derived from the reducible singular fibers of $f$. 
Both of them can be desingularized by successive blow-ups of $X^{(0)}$ as in Section \ref{section: resolution}, and write $X$ and $Y$ for the smooth manifolds obtained from this $\CP^2$-bundle over $S^2$ and $Y^{(0)}$, respectively, in this way. 
It follows from the discussion in Section \ref{section: diffeo type} that 
if $f$ has $n_0$ reducible singular fibers, $Y$ is diffeomorphic to $M \# n_0 \overline{\CP}^2$, obtained by blowing up $M$ at the singular points of the reducible singular fibers. In particular, $Y$ is diffeomorphic to $M$ if $f$ has only irreducible singular fibers. 
Moreover, we see that for the Lefschetz fibration $f \colon M \rightarrow S^2$, 
composing the either map $S^2 \times \CP^1 \dashrightarrow S^2 \times \CP^1$ or $S^2 \tilde{\times}_g \CP^2 \dashrightarrow S^2 \tilde{\times} \CP^1$ with the projection of the $\CP^1$-bundle over $S^2$ yields a singular fibration $p \colon Y \rightarrow S^2$ whose critical values coincide with those of $f$ and singular fiber over $b_j$ is homeomorphic to the singular fiber $f^{-1}(b_j)$ (resp. the total transform of $f^{-1}(b_j)$ under the blow-up at the singular point) if $f^{-1}(b_j)$ is irreducible (resp. reducible).

For future use, we summarize the preceding discussion as follows. 

\begin{theorem}\label{thm: embedding}
Let $f \colon M \rightarrow S^2$ be a hyperelliptic Lefschetz fibration of genus $g \geq 2$ or a non-trivial Lefschetz fibration of genus $1$. Suppose that $f$ has $n_0$ reducible singular fibers. 
Then, there exists a closed $6$-manifold $X$ and its submanifold $Y$ such that: 
\begin{itemize}
\item the manifold $X$ is obtained from the trivial $\CP^2$-bundle $S^2 \times \CP^2$ over $S^2$ (resp. the bundle $S^2 \tilde{\times}_g \CP^2$) by repeated blow-ups if the global braid monodromy associated to $f$ is trivial (resp. non-trivial);
\item the submanifold $Y$ is diffeomorphic to the manifold obtained from $X$ by blowing up at the singular points of the $n_0$ reducible singular fibers; 
\item it also admits a singular fibration $p \colon Y \rightarrow S^2$ whose critical value set agrees with $\Critv(f)$ and singular fiber over $b_j$ is homeomorphic to the singular fiber $f^{-1}(b_j)$ (resp. the total transform of $f^{-1}(b_j)$ under the blow-up) if $f^{-1}(b_j)$ is irreducible (resp. reducible).
\end{itemize}
\end{theorem}

\subsection{Symplectic forms on $X$ making $Y$ symplectic}\label{section: symplectic}

Now for the manifolds $X$ and $Y$ given in Theorem \ref{thm: embedding}, we endow $X$ with a symplectic structure so that $Y$ becomes its symplectic submanifold. 

\begin{proposition}\label{prop: symplectic}
Let $X$ and $Y$ be as above. Then, there exists a symplectic form $\omega$ on $X$ such that $Y$ is a symplectic submanifold of $(X,\omega)$. 
\end{proposition}

\begin{proof}
With the same notation as in Section \ref{section: gluing}, we only deal with the case $X^{(0)}=S^2 \times \CP^2$; 
the other cases proceed similarly. 

Consider the space $Y^{(0)} \subset X^{(0)}$ defined in (\ref{def: Y^{(0)}}). 
We have observed that its singular locus consists of $S^2 \times \{(0:1:0)\}$ and the set $C_0$ of singular points coming from the reducible singular fibers of $f$. 
Blowing up the former singular points $g+1$ times, we obtain the subspace $\tilde{Y}^{(0)}$ in $\tilde{X}^{(0)}=S^2 \times (\CP^2 \# (g+1)\overline{\CP}^2)$ as the proper transform of $Y^{(0)}$. 
Note that $\tilde{Y}^{(0)}$ is smooth away from $C_0$. 
In view of the discussion in Section \ref{section: Kaehler form}, one can take a K\"{a}hler form $\tilde{\omega}_{\mathrm{FS}}$ on $\CP^2 \# (g+1) \overline{\CP}^2$ originating from the Fubini--Study form $\omega_{\mathrm{FS}}$ on $\CP^2$. 
The projection $\tilde{X}^{(0)} \rightarrow S^2$ leads to a singular fibration on $Y^{(0)}$ each fiber of which is a complex curve in $\CP^2 \# (g+1) \overline{\CP}^2$; 
hence it is symplectic with respect to $\tilde{\omega}_{\mathrm{FS}}$. 
Thus, a standard argument (see \cite[pp. 401--403]{GS} for example) shows that for a large constant $K>0$, the $2$-form $K\omega_{0}+\tilde{\omega}_{\mathrm{FS}}$ restricts to a symplectic form on $Y \setminus C_0$. 
Of course, this $2$-form is symplectic on $\tilde{X}^{(0)}$. 

The remaining thing is $C_0$. 
By construction, the space $\tilde{Y}^{(0)}$ is defined to be the zero set of a complex polynomial on a neighborhood of each point of $C_0$. 
Thus, after blowing up $\tilde{X}^{(0)}$ along $C_0$ to get $Y$, it remains a complex submanifold around a neighborhood of the exceptional divisors. 
Now, in light of the discussion in Section \ref{section: Kaehler form} again, we show that $X$ admits a symplectic form $\omega$ obtained from $K\omega_0+\tilde{\omega}_{\mathrm{FS}}$ which is K\"ahler near the exceptional divisors and for which $Y$ is a symplectic submanifold. 
This completes the proof.
\end{proof}

\subsection{Homology class of the submanifold $Y$}\label{section: homology}

Here we examine the homology class of the submanifold $Y$. 
To state the following proposition, we define a distinguished section of the $\CP^2$-bundle $X^{(0)} \rightarrow S^2$. 
Recall that this bundle is given by $X^{(0)}=(D_1(0) \times \CP^2) \cup (D_1(0) \times \CP^2)$ or $X^{(0)}=(D_1(0) \times \CP^2) \cup_{\Phi_{\mathrm{glue}}} (D_1(0) \times \CP^2)$. 
The two copies of $D_1(0) \times \{(1:0:0) \in \CP^2\}$ provide a section $\sigma_0$ of the bundle $X^{(0)} \rightarrow S^2$ in the both cases. 

\begin{proposition}\label{prop: homology}
Given a hyperelliptic Lefschetz fibration $f \colon M \rightarrow S^2$ of genus $g \geq 2$ or genus $1$ with at least one singular fiber, and let $X^{(0)}$, $Y^{(0)}$, $X$ and $Y$ be the spaces defined as in Section \ref{section: gluing}. 
Suppose that $f$ has $n_0$ reducible singular fibers, and the algebraic intersection number of $Y^{(0)}$ and the section $\sigma_0$ of the $\CP^2$-bundle $X^{(0)} \rightarrow S^2$ is $m$ with respect to their orientations induced by the complex structures. 
Then, the homology class of $Y$ is determined by the numbers $g$, $n_0$ and $m$. 
\end{proposition}

\begin{proof}
We will compute the homology class $Y^{(0)}$ in $X^{(0)}$ rather than that of $Y$ in $X$ because the former determines the latter only depending on $g$ and $n_0$ by construction. 
We see that $X^{(0)}$ contains a $\CP^1$-bundle over $S^2$ as a subbundle. 
Indeed, similarly to the definition of $\sigma_0$, if $X^{(0)}=(D_1(0) \times \CP^2) \cup (D_1(0) \times \CP^2)$, the two families of $\{(x:0:z) \in \CP^2\}$ over $D_1(0)$ give a trivial $\CP^1$-bundle over $S^2$ in $X^{(0)}$; 
if $X^{(0)}=(D_1(0) \times \CP^2) \cup_{\Phi_{\mathrm{glue}}} (D_1(0) \times \CP^2)$, the same families glue together via $\Phi_{\mathrm{glue}}$ to form a non-trivial $\CP^1$-bundle over $S^2$. 
Let $A$ and $B$ denote the fourth homology classes of a fiber of $X^{(0)} \rightarrow S^2$ and this $\CP^1$-bundle over $S^2$, respectively. 
Also let $\alpha$ and $\beta$ denote the second homology classes of a line in a fiber of $X^{(0)}$ and the section $\sigma_0$, respectively. 
We easily see that $A$ and $B$ (resp. $\alpha$ and $\beta$) generate $H_4(X^{(0)}; \Z)$ (resp. $H_{2}(X^{(0)}; \Z)$) and 
\begin{align}\label{eqn: intersection}
	A \cdot \alpha = 0, \quad A \cdot \beta = 1, \quad B \cdot \alpha = 1, \quad  B \cdot \beta =0, 
\end{align}
where $\cdot$ denotes the intersection product. 

Suppose that $[Y^{(0)}]$ is written as $[Y^{(0)}] = pA+qB$ for some $p, q \in \Z$. 
In light of (\ref{eqn: intersection}), we have 
\begin{align*}
	[Y^{(0)}] \cdot \alpha  =  (pA+qB) \cdot \alpha = q, \quad [Y^{(0)}] \cdot \beta  =  (pA+qB) \cdot \beta = p. 
\end{align*}
Since the space $Y^{(0)}$ is the zeros of a homogeneous polynomial of degree $2g+2$ over a generic point of $S^2$, we conclude $[Y^{(0)}] \cdot \alpha=2g+2$. 
Moreover, by assumption, $[Y^{(0)}] \cdot \beta=m$. 
Therefore,  we have $[Y^{(0)}]=mA+(2g+2)B$, which finishes the proof. 
\end{proof}

\section{Symplectic submanifolds in a fixed homology class}\label{section: main thm}

\subsection{Twisting submanifolds}\label{section: twisting}

Let $f_1 \colon M_1 \rightarrow S^2$ and $f_2 \colon M_2 \rightarrow S^2$ be hyperelliptic Lefschetz fibrations of genus $g$. 
Take a closed small disk $D$ in $S^2$ containing no critical values of $f_1$ and $f_2$. 
Then, both fibrations are trivialized over $D$, and fix trivializations of them along $\del D$. 
Now identify those fibrations along $\del D$ with the product bundle $\Sigma_g \times S^1 \rightarrow S^1$, where $\Sigma_g$ is a closed, connected, oriented, smooth surface of genus $g$ and $S^1$ is considered as the set $\R/2\pi\Z$. 
Take an element $\varphi$ of the hyperelliptic mapping class group $\H(\Sigma_g)$. 
This induces the fiber-preserving diffeomorphism $F(\varphi) \colon S^1 \times \Sigma_g\rightarrow S^1 \times \Sigma_g$, $(F(\varphi))(\theta, x)=(-\theta, \varphi(x))$, which glues $f_1^{-1}(S^2 \setminus \mathrm{Int}D)$ and $f_2^{-1}(S^2 \setminus \mathrm{Int}D)$ together. 
As a result, we obtain a hyperelliptic Lefschetz fibration $f_1 \#_{\varphi} f_2$ on the glued manifold $M_1 \#_{\varphi} M_2$, called the \textit{fiber sum} of $f_1$ and $f_2$. 
The goal of this section is to interpret the operation of fiber sum in the relative setting.

\begin{lemma}\label{lem: Luttinger}
Let $f_i \colon M_i \rightarrow S^2$ ($i=1,2$) be a hyperelliptic Lefschetz fibration of genus $g \geq 2$ or genus $1$ with at least one singular fiber and $c$ a non-separating simple closed curve on $\Sigma_g$ which is preserved by a hyperelliptic involution. 
Let $(X, \omega)$ and $Y$ denote the symplectic $6$-manifold and its $4$-dimensional symplectic submanifold associated to $f_1 \#_{\mathrm{id}} f_2 \colon M_1 \#_{\mathrm{id}} M_2 \rightarrow S^2$, constructed as in Theorem \ref{thm: embedding} and Proposition \ref{prop: symplectic}. 
Suppose that the global braid monodromy of $f_2$ is trivial.
Then, for any integer $n\in \Z$, there exists a $4$-dimensional symplectic submanifold of $(X,\omega)$ which is diffeomorphic to $M_{1} \#_{\tau_{c}^{n}} M_2$, up to blow-up, and which is homologous to $Y$.
\end{lemma}

To prove this lemma, we introduce a \textit{twisting} construction to modify the submanifold $Y$ symplectically. 
For the trivial fiber sum $f_1 \#_{\mathrm{id}} f_2 \colon M_1 \#_{\mathrm{id}} M_2 \rightarrow S^2$, let $b_j$ ($j=1,\ldots, k_1+k_2$) denote the critical values of $f_1 \#_{\mathrm{id}} f_2$. 
Here, the points $b_1, \ldots, b_{k_1}$ (resp. $b_{k_1+1}, \ldots, b_{k_2}$) are the critical values of $f_1$ (resp. $f_2$). 
We may assume that $b_j=\varepsilon e^{2\pi i j/(k_1+k_2)} \in S^2(=\C \cup \{\infty\})$ for some small $\varepsilon>0$. 
Let $Y_{+}^{(0)}$ denote a subspace of $D_1(0) \times \CP^2$ which is obtained by applying the discussion in Section \ref{section: gluing} to $f_1 \#_{\mathrm{id}} f_2$. 
Choose a smooth embedded loop $\alpha \colon [0,1] \rightarrow D_1(0) \subset S^2$ surrounding only $b_{k_1+1}, \ldots, b_{k_1+k_2}$ among the critical values. 
Since $\alpha$ is a Lagrangian submanifold of $(S^2, \omega_{0})$, there exists a tubular neighborhood $\nu(\alpha)$ of $\alpha$ which is symplectomorphic to $([-r, r] \times S^1, dt \wedge d\theta)$, where $(t, \theta)$ are the standard coordinates on $[-r, r] \times S^1$. 
Since the global braid monodromy of $f_2$ is trivial, 
after applying an isotopy, if necessary, with the symplectic identification $\nu(\alpha) \times \CP^2 \cong [-r, r] \times S^1 \times \CP^2$, we may assume that $Y_{+}^{(0)} \cap ([-r, r] \times S^1 \times \CP^2)$ is given by 
\begin{align}\label{model: trivial}
	\{(t,\theta, (x:y:z)) \mid (x^2-\varepsilon_{1}^2 z^2)(x^{2g}-z^{2g})-y^2z^{2g}=0\}
\end{align}
for some small $\varepsilon_{1}>0$. 
The space $Y_{+}^{(0)} \cap ([-r, r] \times S^1 \times \CP^2)$ is smooth away from $[-r,r] \times S^1 \times \{(0:1:0)\}$. 
Take a smooth function $\chi \colon [-r,r] \rightarrow \R$ such that $\chi(t)=0$ near $-r$ and $\chi(t)=1$ near $r$. 
For a given integer $n$, define $Y_{+}^{(0)}(\alpha, n)$ to be $Y_{+}^{(0)}$ outside $[-r, r] \times S^1 \times \CP^2$, and on $[-r, r] \times S^1 \times \CP^2$
\begin{align}\label{model': trivial}
	\{(t,\theta, (x:y:z)) \mid (x^2-\varepsilon_{1}^2 e^{2n\pi i \chi(t)}z^2)(x^{2g}-z^{2g})-y^2z^{2g}=0\}.
\end{align}
We easily see that the singular point set $Y_{+}^{(0)}(\alpha,n)_{\mathrm{sing}}$ of $Y_{+}^{(0)}(\alpha, n)$ agrees with that of $Y_{+}^{(0)}$, namely $Y_{+,\, \mathrm{sing}}^{(0)}=D_{1}(0) \times \{(0:1:0)\}$. 

So far, simple closed curves on $\Sigma_g$ have not appeared in the above discussion; they come into play from now. 
Take a point $b_0 \in D_1(0)$ away from the disk bounded by $\alpha$ and fix an identification of the fiber of the singular fibration $p \colon Y \rightarrow S^2$ over $b_0$ with $\Sigma_g$. 
We also take a smooth embedded path $\eta$ in $D_1(0) \setminus \{b_1, \ldots, b_{k_1+k_2}\}$ connecting $b_0$ to the point $b'_0=(0,0)$ of $\nu(\alpha)\cong [-r,r] \times S^1$. 
Then, as $p$ is a symplectic fibration along $\eta$, one can define symplectic parallel transport along $\eta$, which sends the simple closed curve $c$ on $\Sigma_{g} \cong F_{b_0}$ in Lemma \ref{lem: Luttinger} to a simple closed curve $c'$ on the fiber $F_{b'_0}$ of $p$ over $b'_0$. 
Note that as $c$ is non-separating, so is $c'$. 
By \cite[Lemma 1]{Fuller} for example, $c'$ is the lift of an embedded arc connecting two branch points of the double cover $F_{b'_0} \rightarrow \CP^1$ induced by the double branched cover $Y \rightarrow Bl(S^2 \times \CP^1)$ or $Y \rightarrow Bl(S^2 \tilde{\times} \CP^1)$. 
In (\ref{model: trivial}), the map $\{(0,0)\} \times (\CP^2 \setminus \{(0:1:0)\}) \rightarrow \CP^1$, $(0,0,(x:y:z)) \mapsto (x:z)$ defines the double cover of its image with branch points $(\pm\varepsilon_1:1)$, $(e^{2\pi i j/2g} : 1)$ ($j=1,\ldots, 2g$). 
Thus, with the help of isotopy, we may assume $c'$ to be the lift of the embedded arc $[0,1] \ni t \mapsto ((2t-1)\varepsilon_1:1) \in \CP^1$ connecting $(\pm\varepsilon_1:1)$; this is the simple closed curve given as the set of points $((0,0), ((2t-1)\varepsilon_1:y:1)) \in \{(0,0)\} \times \CP^2$ with $t \in [0,1]$ satisfying the defining equation of (\ref{model: trivial}). 
As a consequence, topologically the space (\ref{model': trivial}) can be regarded as the one obtained by changing (\ref{model: trivial}) near $c'$. 
Now that $c$ has been related to the discussion, let $Y_{+}^{(0)}(\alpha,c,n)$ denote $Y_{+}^{(0)}(\alpha, n)$.

\begin{lemma}\label{lem: symplectic}
If the space $Y_{+}^{(0)}$ is a symplectic submanifold of $(D_1(0)\times \CP^2, K\omega_{0}+\omega_{\mathrm{FS}})$ away from $Y_{+,\, \mathrm{sing}}^{(0)}=D_{1}(0) \times \{(0:1:0)\}$ for some $K>0$, then so is $Y_{+}^{(0)}(\alpha, c, n)$ away from $Y_{+}^{(0)}(\alpha, c, n)_{\mathrm{sing}}$. 
\end{lemma}

\begin{proof}
Set $Y_{+}^{(0)}(\alpha, c, n)_{\mathrm{sm}}=Y_{+}^{(0)}(\alpha, c, n) \setminus Y_{+}^{(0)}(\alpha, c, n)_{\mathrm{sing}}$.
Since $Y_{+}^{(0)}(\alpha, c, n)_{\mathrm{sm}}$ agrees with $Y_{+}^{(0)} \setminus Y_{+,\, \mathrm{sing}}^{(0)}$ outside $\nu(\alpha) \times \CP^2$, it suffices to prove the lemma on $\nu(\alpha) \times \CP^2$. 
Moreover, when $z=0$, the space $Y_{+}^{(0)}(\alpha, c, n)$ over $\nu(\alpha)$ is included in $Y_{+}^{(0)}(\alpha, c, n)_{\mathrm{sing}}$. 
Thus, we will show that $Y_{+}^{(0)}(\alpha, c, n)_{\mathrm{sm}}$ is a symplectic submanifold of $(\nu(\alpha) \times \C^2, \omega_{0}+\omega_{\mathrm{FS}}|_{\C^2})$, regarding $\C^2$ as $\CP^2 \setminus \{z = 0\}$. 

Let $g(t,x,y,z)$ denote the function defining $Y_{+}^{(0)}(\alpha, c, n)$ in (\ref{model': trivial}). 
The intersection $Y_{+}^{(0)}(\alpha, c, n) \cap (\nu(\alpha) \times \C^2)$ is written as $\{g(t, x,y,1)=0\} \subset [-r,r] \times S^1 \times \C^2$. 
The non-degeneracy of the Jacobi matrix $(\del g/ \del x, \del g/\del y)$ at any point $(t,\theta, x,y)$ shows that for a fixed $(t, \theta) \in [-r, r] \times S^1$, the zero set $F_{(t,\theta)} \coloneqq \{g(t,x,y,1)=0\}$ is a complex submanifold of $\{(t,\theta) \} \times \C^2$. 
Moreover, it is a symplectic submanifold of $([-r,r] \times S^1 \times \C^2, Kdt \wedge d\theta+\omega_{\mathrm{FS}}|_{\C^2})$. 
At any point $p_0=(t,\theta, x,y)$, we have the splitting 
\[
	T_{p_0}Y_{+}^{(0)}(\alpha, c, n)_{\mathrm{sm}}=T_{p_0}F_{(t,\theta)} \oplus (T_{p_0}Y_{+}^{(0)}(\alpha, c, n)_{\mathrm{sm}} \cap (T_{p_0}F_{(t,\theta)})^{\perp Kdt \wedge d\theta +\omega_{\mathrm{FS}}|_{\C^2}}), 
\]
where the symplectic complement is taken in $T_{p_0}([-r,r] \times S^1 \times \C^2)$. 
Thus, all we have to prove is that the second summand on the right-hand side is symplectic.

Write $x=x_1+ix_2$ and $y=y_1+iy_2$ with real numbers $x_1,x_2,y_1,y_2$ and define real-valued functions $\xi(t,x_1, x_2, y_1,y_2)$ and $\eta(t, x_1, x_2, y_1,y_2)$ by 
\[
	g(t, x,y,1)=\xi(t, x_1, x_2, y_1,y_2)+i\eta(t, x_1, x_2, y_1,y_2). 
\]
Note that the tangent space $T_{p_0}Y_{+}^{(0)}(\alpha, c, n)_{\mathrm{sm}}$ agrees with the kernel of the differential $Dg$ at $p_0$, expressed by the matrix
\[
	\left(
   \begin{array}{cccccc}
	\del \xi/\del t & \del \xi/ \del \theta & \del \xi/\del x_1  & \del \xi/ \del x_2 & \del \xi/ \del y_1 & \del \xi/\del y_2 \\
       \del \eta/\del t & \del \eta/ \del \theta & \del \eta/\del x_1  & \del \eta/ \del x_2 & \del \eta/ \del y_1 & \del \eta/\del y_2
    \end{array}
\right).
\]
The non-degeneracy of $(\del g/ \del x, \del g/\del y)$ implies the non-degeneracy of 
\[
	\left(
   \begin{array}{cccc}
    \del \xi/\del x_1  & \del \xi/ \del x_2 & \del \xi/ \del y_1 & \del \xi/\del y_2\\
     \del \eta/\del x_1  & \del \eta/ \del x_2 & \del \eta/ \del y_1 & \del \eta/\del y_2
    \end{array}
\right), 
\]
which leads to the existence of a vector $v$ in $T_{p_0}(\{(t,\theta)\} \times \C^2)\cap (T_{p_0}F_{(t,\theta)})^{\perp Kdt \wedge d\theta +\omega_{\mathrm{FS}}|_{\C^2}}$ such that $\del/\del t+v \in T_{p_0}Y_{+}^{(0)}(\alpha, c, n)_{\mathrm{sm}}$. 
Moreover, $\xi$ and $\eta$ are independent of $\theta$, we see that $\del/\del \theta \in T_{p_0}Y_{+}^{(0)}(\alpha, c, n)_{\mathrm{sm}}$. 
The two vectors $\del/\del t+v$ and $\del/ \del \theta$ lie in the symplectic complement $(T_{p_0}F_{(t,\theta)})^{\perp Kdt \wedge d\theta +\omega_{\mathrm{FS}}|_{\C^2}}$ and satisfy 
\[
	(Kdt \wedge d\theta +\omega_{\mathrm{FS}}|_{\C^2})(\del/\del t+v, \del/\del \theta)=K>0. 
\]
This completes the proof. 
\end{proof}

\begin{proof}[Proof of Lemma \ref{lem: Luttinger}]
In this proof, we use the same notations as above. 
Let $Y(\alpha, c, n)$ be the submanifold of $X$ obtained from $Y^{(0)}(\alpha, c, n)$ by blow-ups as in Section \ref{section: gluing}. 
It follows easily from Lemma \ref{lem: symplectic} that $Y(\alpha, c, n)$ is a symplectic submanifold of $(X,\omega)$. 

 Since $Y^{(0)}(\alpha, c, n)_{\mathrm{sing}}=Y^{(0)}_{\mathrm{sing}}$, the submanifold $Y(\alpha, c, n)$ can be also seen as a submanifold obtained from $Y$ by changing it only over $\nu$. 
We examine the effect of this change on the monodromy. 
Recall that by construction, $Y$ and $Y(\alpha, c, n)$ are fiber bundles over $S^2 \setminus \{b_1, \ldots, b_{k_1+k_2}\}$. 
Take an embedded path $\gamma$ connecting $b_0$ to an arbitrary point $b_i \in \{b_{k_1+1}, \ldots, b_{k_1+k_2}\}$ and intersecting $\alpha$ at one point. 
Let $\hat{\gamma}$ be a loop based at $b_0$ associated to $\gamma$ as in Section \ref{section: LF}. 
Let $p' \colon Y' \rightarrow S^2$ and $p'_n \colon Y'(\alpha, c, n) \rightarrow S^2$  denote Lefschetz fibrations obtained from $p \colon Y \rightarrow S^2$ and $p_n \colon Y(\alpha, c, n) \rightarrow S^2$, respectively, by repeated blow-downs of $(-1)$-spheres derived from the reducible singular fibers, where $p_n \colon Y(\alpha, c, n) \rightarrow S^2$ is the singular fibration naturally derived from $p$. 
Suppose that the monodromy of $p' \colon Y' \rightarrow S^2$ along $\hat{\gamma}$ is given by $\tau_{c_i}$. 
Then, in light of the choice of the local model (\ref{model: trivial}), we find that the monodromy of $p'_n$ along $\hat{\gamma}$ agrees with $\tau_{\tau_c^{n}(c_i)}$. 
Hence, $p'$ and $p'_n$ have collections of vanishing cycles 
\[
	(c_1, \ldots, c_{k_1}, c_{k_1+1}, \ldots, c_{k_1+k_2}) 
\textrm{\ and\ } 
	(c_1, \ldots, c_{k_1}, \tau_{c}^n(c_{k_1+1}), \ldots, \tau_c^{n}(c_{k_1+k_2})),
\]
respectively, for an appropriately chosen basis for $\pi_1(S^2 \setminus \{b_1, \ldots, b_{k_1+k_2}\}, b_0)$. 
Therefore, by the aforementioned result of Kas \cite{Kas}, $Y'(\alpha, c, n)$ is diffeomorphic to $M_1 \#_{\tau^n_{c}} M_2$, which shows that $Y(\alpha, c, n)$ is diffeomorphic to a blow-up of $M_1 \#_{\tau^n_{c}} M_2$.

We next see the homology class of $Y(\alpha, c, n)$. 
By Proposition \ref{prop: homology}, it is sufficient to check that the algebraic intersection number of $Y^{(0)}$ and the section $\sigma_0$ is equal to that of $Y^{(0)}(\alpha, c, n)$ and $\sigma_0$. 
The section $\sigma_0$ agrees with the map $s \mapsto (s, (1:0:0))$ when restricted to $\nu$, so it intersects with neither $Y^{(0)}$ nor  $Y^{(0)}(\alpha, c, n)$ on $\nu(\alpha) \subset S^2$. Thus, this proves $[Y(\alpha, c, n)]=[Y] \in H_4(X; \Z)$. 
\end{proof}

We say a submanifold is obtained by the \textit{$n$-fold twisting construction} for the pair $(Y,c)$ over $\alpha$ if it is obtained from $Y$ by replacing $Y^{(0)}$ with the twisted singular subspace $Y^{(0)}(\alpha, c, n)$ as in the above proof. 

\subsection{Lefschetz fibrations on noncomplex smooth $4$-manifolds}\label{section: OS}

Let $\Sigma_2$ be a closed, connected, oriented, smooth surface of genus $2$ and let $c, c_1, \ldots, c_4$ be the simple closed curves on $\Sigma_2$ depicted as in Figure \ref{fig: SCC}. 
Notice that as 
\begin{align*}
\M(\Sigma_2)=\mathcal{H}(\Sigma_2), 
\end{align*}
every Lefschetz fibration of genus $2$ is hyperelliptic. 
Matsumoto \cite{Mat} gave a Lefschetz fibration $f_0 \colon M_0 \rightarrow S^2$ of genus $2$ whose collection of vanishing cycles is
\[
	(c_1,\ldots, c_4, c_1, \ldots , c_4). 
\]
This led Ozbagci and Stipsicz \cite{OS00} to a family of genus-$2$ Lefschetz fibrations whose total spaces do not admit complex structures. 
We consider a variant of their family given by 
\[
	(f_0\#_{\mathrm{id}} f_0) \#_{\tau^n_c}(f_0\#_{\mathrm{id}} f_0) \colon  (M_0 \#_{\mathrm{id}} M_0) \#_{\tau^n_c} (M_0 \#_{\mathrm{id}} M_0) \rightarrow S^2 
\]
for every $n \in \Z_{\geq 0}$. 
For convenience, we set 
\begin{align}\label{eqn: Mn}
	M(n)= (M_0 \#_{\mathrm{id}} M_0) \#_{\tau^n_c} (M_0 \#_{\mathrm{id}} M_0) \quad \textrm{and} \quad g_n=(f_0\#_{\mathrm{id}} f_0) \#_{\tau^n_c}(f_0\#_{\mathrm{id}} f_0).
\end{align}
The collection of vanishing cycles of each Lefschetz fibration $g_n$ consists of $32$ simple closed curves 
\begin{gather*}
	c_1,\ldots, c_4, c_1, \ldots , c_4, c_1,\ldots, c_4, c_1, \ldots , c_4, \\
	\tau_{c}^n(c_1), \ldots, \tau_{c}^n(c_4), \tau_{c}^n(c_1), \ldots, \tau_{c}^n(c_4),\tau_{c}^n(c_1), \ldots, \tau_{c}^n(c_4), \tau_{c}^n(c_1), \ldots, \tau_{c}^n(c_4).
\end{gather*}
The next proposition essentially follows from \cite{OS00}. 

\begin{figure}[ht]
\vspace{11pt}
		\begin{overpic}[width=180pt,clip]{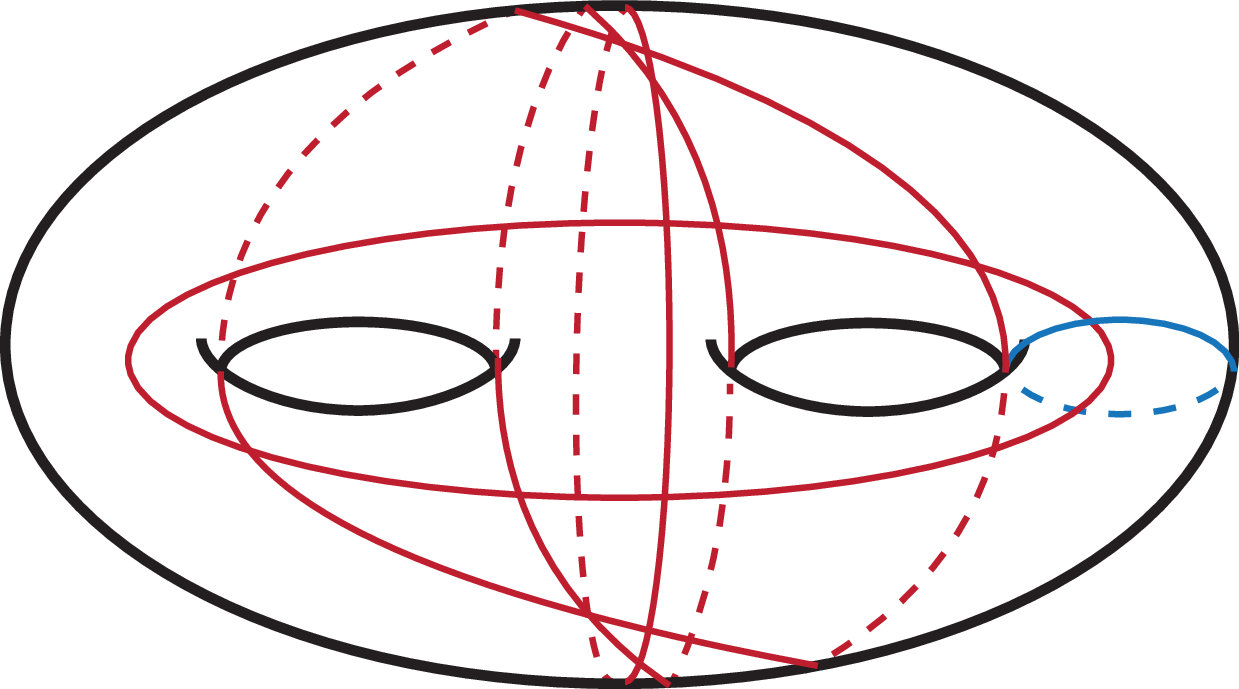}
	\put(7,46){$c_1$}
	\put(86,46){$c_2$}
	\put(68.5,18){$c_{3}$}
	\put(52,12){$c_4$}
	\put(168,55){$c$}
	\put(168,85){$\Sigma_2$}
	\end{overpic}
	\caption{Simple closed curves $c, c_1, \ldots, c_4$ on $\Sigma_2$.}
	\label{fig: SCC}
\end{figure}
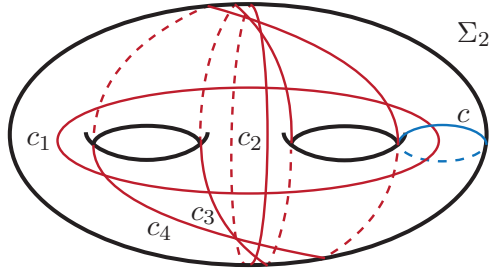

\begin{proposition}\label{prop: Mn}
For the $4$-manifold $M(n)$ ($n>0$), we have the following. 
\begin{enumerate}
\item $\pi_1(M(n)) \cong \Z \oplus \Z_n$. 
\item The $4$-manifold $M(n)$ does not admit a complex structure, and neither does a blow-up of $M(n)$. 
\end{enumerate}
\end{proposition}

\begin{proof}
To prove the first assertion, recall the fact that the fundamental group of the total space of a Lefschetz fibration is isomorphic to the quotient of that of a regular fiber by the normal subgroup generated by all the vanishing cycles if the fibration has a section \cite[Lemma 3.2]{ABKP}. 
It is known that $f_0$ has a section (see e.g. \cite[Remark 2.8]{OS04}), and hence so does each $g_n$. 
Thus, we have 
\[
	\pi_1(M(n)) \cong \pi_1(\Sigma_2)/ \langle c_1, c_2, c_3, c_4, \tau_c^n(c_1), \tau_c^n(c_2), \tau_c^n(c_3), \tau_c^n(c_4) \rangle,
\]
which is isomorphic to $\Z \oplus \Z_n$ by \cite[Theorem 1.2]{OS00}. 

To prove the second assertion, we first compute $b_2^{+}(M(n))$, the number of positive eigenvalues of the intersection form on $H_2(M(n); \Z)$. 
A handlebody structure on $M(n)$ associated to the fibration $g_n$ proves the Euler characteristic $\chi(M(n))=28$. 
Since $M_0$ is known to be diffeomorphic to $(S^2 \times T^2 ) \# 4\overline{\CP}^2$, the Novikov additivity shows the signature $\sigma(M(n))=-16$. 
Thus, the identity $\sigma(M(n))+\chi(M(n))=2(1-b_1(M(n))+b_2^{+}(M(n)))$ concludes $b_2^{+}(M(n))=6$.
The proof of \cite[Theorem 1.3]{OS00} actually shows that a smooth, closed, connected $4$-manifold $M$ with $b_2^{+}(M)\geq 1$ and $\pi_1(M) \cong \Z \oplus \Z_n$ does not admit any complex structure. 
Therefore, combining this result with the fact that a blow-up preserves $\pi_1$ and $b_2^{+}$ implies the second assertion. 
\end{proof}

Let us compute the global braid monodromies associated to $g_{n}$. 
Since 
\begin{align*}
(\tau_{c_1} \cdots \tau_{c_4})^2=1, 
\end{align*}
the distinguished lift $\hat{\tau}$ of $(\overline{\tau}_{c_1} \cdots \overline{\tau}_{c_4})^2 \in \M(S^2, 2g+2)$ is either $1$ or $(\hat{\sigma}_1 \cdots \hat{\sigma}_5)^6$. 
Thus, the global braid monodromy associated to the fibration $f_0 \#_{\mathrm{id}} f_0$, given by the lift of $(\overline{\tau}_{c_1} \cdots \overline{\tau}_{c_4})^4$, is $\hat{\tau}^2=1$. 
Furthermore, simultaneous conjugation does not affect the global braid monodromy, and the lift of $(\overline{\tau}_{\tau^n_c(c_1)} \cdots \overline{\tau}_{\tau^n_c(c_4)})^4$ coincides with $\hat{\tau}^2=1$ for any $n \in \Z_{> 0}$. 
This leads to the triviality of the global braid monodromy associated to each $g_n$.

\subsection{Proof of the main theorem and its corollary}\label{section: proof of main thm}

\begin{proof}[Proof of Theorem \ref{thm: main}]
By Theorem \ref{thm: embedding}, to the fiber sum $g_0 \colon M(0) \rightarrow S^2$ of two copies of $f_0 \#_{\mathrm{id}} f_0$ defined by (\ref{eqn: Mn}), we associate a $4$-dimensional symplectic submanifold $Y$ of a blow-up $(X, \omega)$ of a $\CP^2$-bundle over $S^2$. 
In fact, this $\CP^2$-bundle is trivial since the global monodromy associated to $g_0$ is trivial. 
To obtain the desired family of symplectic submanifolds, take a smooth embedded loop $\alpha$ in $S^2$ which surrounds the critical values of the latter $f_0 \#_{\mathrm{id}} f_0 \colon M_0 \#_{\mathrm{id}} M_0 \rightarrow S^2$ of the fiber sum $g_0$ and which is related to $c$ as in Section \ref{section: twisting}.  
Since the global braid monodromy associated to it is trivial, we can perform the $n$-fold twisting construction for the pair $(Y,c)$ over $\alpha$. 
Then, according to Lemma \ref{lem: Luttinger}, the resulting submanifold $Y(\alpha, c, n)$ is symplectic in $(X, \omega)$ and diffeomorphic to a blow-up of $(M_0 \#_{\mathrm{id}} M_0) \#_{\tau^n_c} (M_0 \#_{\mathrm{id}} M_0) $, i.e., $M(n)$. 
Proposition \ref{prop: Mn} shows that each $Y(\alpha, c, n)$ ($n>0$)  has the fundamental group $ \Z \oplus \Z_n$ and does not admit a complex structure. 
In particular, the submanifolds $Y(\alpha, c, n)$'s are mutually homotopy inequivalent.

To finish the proof, we verify the simply connectedness of $X$. 
Recall the construction of $X$: 
it is obtained from the simply connected $6$-manifold $S^2 \times \CP^2$ by successive blow-ups at points and submanifolds diffeomorphic to $S^2$. 
Since the blow-up operation does not affect the fundamental group, we conclude that $X$ is simply connected. 
\end{proof}

\begin{proof}[Proof of Corollary \ref{cor: main}]
Let us denote the closed symplectic $6$-manifold obtained in the proof of Theorem \ref{thm: main} by $(X,\omega)$ and the symplectic submanifold $Y(\alpha, c, n)$ of $(X,\omega)$ by $Y_n$. 
Then, consider the product symplectic manifold 
\[
	(X \times (S^2)^m, \omega+\Omega),  
\]
where $\Omega$ is a symplectic form on the $m$-fold product $(S^2)^m$ of $S^2$. 
As $X$ is simply connected, so is $X \times (S^2)^m$.
The submanifolds $Y_n$'s are homotopy inequivalent but homologous, so the family $\{Y_n \times (S^2)^m\}_{n \in \Z_{>0}}$ of symplectic submanifolds in $(X \times (S^2)^m, \omega+\Omega)$ is what we desired.
\end{proof}
\appendix

\section{Symplectic submanifolds of $(\CP^3,\omega_{\mathrm{FS}})$}\label{appendix}

In this appendix, we will give an explicit proof of the following fact, which is well known (see e.g. \cite[Remark 4.8]{Li}) but whose proof has not appeared in the literature. 

\begin{theorem} \label{thm: uniquness}
Let $Y$ be a connected symplectic submanifold of $(\CP^3, \omega_{\mathrm{FS}})$ of dimension $4$ in the homology class $k[\CP^2] \in H_4(\CP^3; \Z)$. 
If $1 \leq k \leq 3$, the submanifold $Y$ is diffeomorphic to a smooth projective hypersurface of degree $k$. 
\end{theorem}

Recall that a smooth projective hypersurface of degree $1$ is diffeomorphic to $\CP^2$; one of degree $2$ is diffeomorphic to $S^2 \times S^2$; one of degree $3$ is diffeomorphic to $\CP^2 \# 6 \overline{\CP}^2$.

The proof we give in this appendix relies on a result about $4$-dimensional symplectic manifolds.  

\begin{theorem}[{Ohta and Ono \cite[Theorem 1.3]{OO_SW2}}]\label{thm: OO}
Let $(M, \omega)$ be a closed connected symplectic $4$-manifold with $c_1(TM)=\lambda[\omega]$ for some constant $\lambda>0$. 
Then, $M$ is diffeomorphic to a Del Pezzo surface.
\end{theorem}

A \textit{Del Pezzo surface} is a smooth complex projective manifold with the ample anticanonical bundle. 
The possible diffeomorphism types of Del Pezzo surfaces are completely understood: blow-ups of $\CP^2$ at $m$ points ($0 \leq m \leq 8$) or $S^2 \times S^2$.

\begin{proof}[Proof of Theorem \ref{thm: uniquness}]
Since $Y$ is symplectic, one can take an $\omega_{\mathrm{FS}}$-tame almost complex structure $J$ on $\CP^3$ which makes $Y$ an almost complex submanifold. 
Then, the tangent bundle $TY$ of $Y$ and the normal bundle $N_{Y/\CP^3}$ in $\CP^3$ become complex vector bundles. 
The splitting 
\[
T\CP^3|_{Y} \cong TY \oplus N_{Y/\CP^3}
\] 
yields the identity between Chern classes of vector bundles: 
\[
	c(T\CP^3|_{Y})=(1+c_1(TY)+c_2(TY))(1+c_1(N_{Y/\CP^3})).
\]
Set $\eta=[\omega_{\mathrm{FS}}|_{Y}] \in H^{2}(Y; \Z)$. 
As $c_1(T\CP^3)=[\omega_{\mathrm{FS}}]$, we have 
\[
c(T\CP^3|_{Y})=(1+\eta)^4.
\] 
Moreover, since $[Y]$ is Poincar\'e dual to $k[\omega_{\mathrm{FS}}]$, the first Chern class $c_1(N_{Y/\CP^3})$ agrees with $k\eta$. 
Thus, 
\[
	c_1(TY)=(4-k)\eta, \quad c_2(TY)=(k^2-4k+6)\eta^2.
\]
In light of Theorem \ref{thm: OO} combined with the above first equality, we conclude that $Y$ is diffeomorphic to a Del Pezzo surface.

To determine the diffeomorphism type of $Y$, we first specify $b_2^{-}(Y)$. 
Since $Y$ is a Del Pezzo surface, $b_1(Y)=b_3(Y)=0$ and $b_2^{+}(Y)=1$. 
Applying the identity 
\[
	c_1^2(TY)=3\sigma(Y)+2\chi(X)=3\sigma(Y)+2\int_{X}c_2(X), 
\]
we find 
\[
	b_2^{-}(Y)=\frac{3(k-1)^2}{k^2-6k+11}=
	\begin{cases}
	0 & (k=1),\\
	1 & (k=2),\\
	6 & (k=3),
	\end{cases}
\]
where $\sigma(Y)$ and $\chi(Y)$ denote the signature and the Euler characteristic of $Y$, respectively. 
Thus, we conclude that $Y$ is diffeomorphic to $\CP^2$ if $k=1$ and diffeomorphic to $\CP^2 \# 6 \overline{\CP^2}$ if $k=3$. 
In the case $k=2$, we still have two possibilities, namely, $S^2 \times S^2$ and $\CP^2 \# \overline{\CP}^2$. 
However, it follows from $b_1(Y)=0$ and $c_1(TY)=2\eta$ that $Y$ carries a spin structure. 
Therefore, $Y$ is diffeomorphic to $S^2 \times S^2$. 
This completes the proof. 
\end{proof}

We remark that for $k=1,2$, the theorem was also proven by Hind \cite{Hind03} via pseudo-holomorphic curves.

\subsection*{Acknowledgements}
The author would like to thank Kimihiko Motegi for asking him an interesting question, which started this work, and Jorge Vit\'{o}rio Pereira for the e-mail correspondence on Totaro's result.  
This work was supported by JSPS KAKENHI Grant Numbers 20K22306, 22K13913.


\end{document}